\documentclass[a4paper,draft,reqno,12pt]{amsart}
\usepackage[utf8]{inputenc}
\usepackage[T1, T2A]{fontenc}
\usepackage[english]{babel}
\usepackage{amsmath}
\usepackage{amssymb}
\usepackage{amscd}
\usepackage{amsthm}
\usepackage{euscript}
\usepackage{xcolor}
\usepackage[cmtip,arrow]{xy}
\usepackage{pb-diagram, pb-xy}
\usepackage{tikz}
\newtheorem{proposition}{Proposition}

\newtheorem{lemma}{Lemma}
\newtheorem{theorem}{Theorem}

\newtheorem{corollary}{Corollary}
\theoremstyle{definition}
\newtheorem{definition}{Definition}
\newtheorem{example}{Example}
\newtheorem{construction}{Construction}
\theoremstyle{remark}
\newtheorem {remark}{Remark}

\DeclareMathOperator{\Spec}{Spec}
\DeclareMathOperator{\Hom}{Hom}
\DeclareMathOperator{\cone}{cone}
\DeclareMathOperator{\Span}{span}
\DeclareMathOperator{\Cl}{Cl}

\DeclareMathOperator{\Div}{div}

\DeclareMathOperator{\Ker}{Ker}

\newcommand{\GG}{\mathbb{G}}
\newcommand{\NN}{\mathbb{N}}
\newcommand{\KK}{\mathbb{K}}

\newcommand{\ZZ}{\mathbb{Z}}
\newcommand{\RRR}{\mathcal{R}}
\newcommand{\QQ}{\mathbb{Q}}
\renewcommand{\AA}{\mathbb{A}}

\renewcommand{\phi}{\varphi}
\renewcommand{\ge}{\geqslant}
\renewcommand{\le}{\leqslant}

\newcommand{\Zgezero}{\mathbb{Z}_{\geqslant 0}}
\newcommand{\pa}{\partial}
\newcommand{\overbar}[1]{\mkern 2.5mu\overline{\mkern-2.5mu#1\mkern-0.5mu}\mkern 0.5mu}
\newcommand{\barX}{\overbar{X}}
\newcommand{\barM}{\overbar{M}}
\newcommand{\barT}{\overbar{T}}
\newcommand{\barG}{\overbar{G}}

\begin{document}
\date{}
\title[Commutative algebraic monoids on affine surfaces]{Commutative algebraic monoid structures \\ on affine surfaces}

\author{Sergey Dzhunusov}
\address{National Research University Higher School of Economics, Faculty of Computer Science, Pokrovsky boulevard 11, Moscow, 109028 Russia}
\email{dzhunusov398@gmail.com}

\author{Yulia Zaitseva}
\address{National Research University Higher School of Economics, Faculty of Computer Science, Pokrovsky boulevard 11, Moscow, 109028 Russia}
\email{yuliazaitseva@gmail.com}



\thanks{The work was supported by the Foundation for Advancement of Theoretical Physics and Mathematics ``BASIS''}

\begin{abstract}
We classify commutative algebraic monoid structures on normal affine surfaces over an algebraically closed field of characteristic zero. The answer is given in two languages: comultiplications and Cox coordinates. The result follows from a more general classification of commutative monoid structures of rank~$0$, $n-1$ or~$n$ on a normal affine variety of dimension~$n$. 
\end{abstract}

\maketitle


\section*{Introduction}
\label{introd_sec}

An (affine) algebraic monoid is an irreducible (affine) algebraic variety $X$ with an associative multiplication
\[
\mu \colon X\times X\to X,\quad (x,y)\mapsto x*y,
\]
which is a morphism of algebraic varieties and admits an element $e\in X$ such that ${e*x=x*e=x}$ for all $x\in X$. Examples of affine algebraic monoids are affine algebraic groups and multiplicative monoids of finite-dimensional associative algebras with unit. 

The group of invertible elements $G(X)$ of an algebraic monoid $X$ is an algebraic group, which is Zariski open in $X$. By~\cite[Theorem~3]{Ri2}, every algebraic monoid $X$ whose group of invertible elements $G(X)$ is an affine algebraic group is an affine monoid. 

By a group embedding we mean an irreducible affine variety $X$ with an open embedding $G\hookrightarrow X$ of an affine algebraic group $G$ such that both actions by left and right multiplications of $G$ on itself can be extended to actions of $G$ on $X$. In other words, the variety $X$ is a $(G\times G)$-equivariant open embedding of the homogeneous space $(G\times G)/\Delta(G)$, where $\Delta(G)$ is the diagonal in $G\times G$. For an affine variety $X$ and an affine algebraic group $G$, there is a natural correspondence between group embeddings $G\hookrightarrow X$ and monoid structures on~$X$ with $G(X) = G$. This is proved in~\cite[Theorem~1]{Vi} for characteristic zero and in~\cite[Proposition~1]{Ri1} for the general case. 

The theory of affine algebraic monoids and group embeddings is a rich and deeply developed area of mathematics lying at the intersection of algebra, algebraic geometry, combinatorics and representation theory; see~\cite{Pu5,Re,Ri1,Vi} for general presentations. An affine algebraic monoid $X$ is called reductive if the group $G(X)$ is a reductive affine algebraic group. The most developed is the theory of reductive monoids, see e.g. the combinatorial classification of reductive monoids in~\cite{Vi,Ri1}. 

Let the ground field $\KK$ be algebraically closed and of characteristic zero. We study commutative affine monoids, i.e. affine monoids $X$ with commutative multiplication $\mu$ or, equivalently, commutative group of invertible elements~$G(X)$. The intersection of commutative and reductive cases are affine torus embeddings, i.e. affine toric varieties, see~\cite{Neeb} for more information about toric monoids. In the commutative case, the group $G(X)$ splits into the direct product $\GG_m^r \times \GG_a^s$, where $\GG_m = (\KK^\times, \times)$ and $\GG_a = (\KK, +)$ are the multiplicative and the additive groups of the ground field $\KK$, respectively. We call the number~$r$ the rank of the commutative monoid~$X$. 

In~\cite{ABZ}, commutative monoid structures on affine spaces $\AA^n$ are studied. In particular, \cite[Proposition~1]{ABZ} gives the classification for an arbitrary dimension~$n$ and ranks $0$,~$n-1$ and~$n$. This implies the classification of monoids on~$\AA^1$ and $\AA^2$, and \cite[Theorem~1]{ABZ} gives the classification in dimension~$3$. 

In this paper, we obtain a classification of commutative monoid structures on normal affine surfaces (Section~\ref{surf_sec}). For that, we generalize some results of~\cite{ABZ} to arbitrary normal affine varieties. It turns out that any affine algebraic variety admitting a commutative monoid structure of rank $0$, $n-1$ or $n$ is toric, and structures of rank~$n-1$ are described by Demazure roots of the variety. We give precise formulas in two languages: the first one describes monoid structures via comultiplications $\mu^*\colon \KK[X] \to \KK[X] \otimes \KK[X]$ (Section~\ref{comult_sec}), and the second one is the lifting of monoid structures from $X$ to the total coordinate space~$\barX$ via the Cox construction (Section~\ref{Coxlang_sec}). Namely, along with the natural addition on affine spaces and comultiplications $\mu^*\colon\chi^u \mapsto \chi^u \otimes \chi^u$ provided by the toric structure, we obtain a family of comultiplications $\mu^*\colon\chi^u \mapsto \chi^u \otimes \chi^u (1 \otimes \chi^e + \chi^e \otimes 1)^{\langle p, u\rangle}$ given by Demazure roots~$e$ corresponding to the primitive vector~$p$ on a ray of the fan of the toric variety. 

The authors are grateful to their supervisor Ivan Arzhantsev for careful reading of the paper and valuable suggestions, to Juergen Hausen for useful comments, to the referee for proposed corrections and improvements, and to Boris Bilich who pointed our attention to an innacuracy in the formulation of Corollary~1. 

\section{Toric varieties and $\GG_a$-actions}
\label{toric_sec}

Let $X$ be a normal irreducible affine algebraic variety with an action of an algebraic torus $T = (\KK^\times)^n$. Denote by $M$ the character lattice of the torus $T$ and by $N$ the lattice of one-parameter subgroups. The dual action of the torus $T$ on the algebra $\KK[X]$ of regular functions on $X$ defines the decomposition of $\KK[X]$ into the direct sum 
$\KK[X] = \bigoplus\limits_{u \in M} \KK[X]_u$, where the torus $T$ acts on the subspace $\KK[X]_u$ by multiplication by the character $\chi^u(t)=t^u$, $t \in T$, corresponding to $u \in M$. Denote $S_X = \{u \in M \mid \KK[X]_u \ne 0\}$. 

Let $X$ be a \emph{toric variety}, that is the torus $T$ acts effectively on $X$ with an open orbit. In this case, the semigroup $S_X$ equals $M \cap \omega$ for some rational polyhedral cone $\omega \subseteq M \otimes_\ZZ \QQ$ and $\dim \KK[X]_u = 1$ for any $u \in S_X$. Fixing a point $x_0$ in the open orbit, we can identify $T$ with an open orbit via the orbit map $T \to X$, $t \mapsto tx_0$, whence a character $\chi^u\colon T \to \KK^\times$ can be identified with a rational function $\chi^u \in \KK(X)$, $u \in M$. Then we have a decomposition 
\begin{equation}
\label{torgrad_eq}
\KK[X] = \bigoplus\limits_{u \in S_X} \KK\chi^u.
\end{equation}

For an algebra~$A$, a linear map $\pa\colon A \to A$ is a \emph{derivation} if it satisfies the Leibniz rule: $\pa(fg) = \pa(f)g + f \pa(g)$ for any $f, g \in A$. A derivation $\pa$ is called \emph{locally nilpotent} (LND) if for any $f \in A$ there exists a number $n \in \NN$ such that $\pa^n(f) = 0$. Let $\GG_a = (\KK, +)$ be the additive group of the ground field~$\KK$. The exponential map defines a bijection between locally nilpotent derivations on an algebra $A$ and rational $\GG_a$-actions on~$A$: $s \cdot f = \exp(s\pa)(f)$ for $s \in \GG_a$, $f \in A$. For an affine algebraic variety~$X$ we obtain that LNDs on the algebra $\KK[X]$ are in bijection with $\GG_a$-actions on~$X$, see~\cite[Section~1.5]{Fr}. 

Let $A = \bigoplus\limits_{u \in S} A_u$ be graded by a semigroup $S$. A derivation $\pa\colon A \to A$ is called \emph{homogeneous} if it maps homogeneous elements to homogeneous ones. From the Leibniz rule it follows that a homogeneous derivation has the \emph{degree} $\deg \pa \in \ZZ S$ such that $\pa(A_u) \subseteq A_{u + \deg \pa}$ for any $u \in S$. It is easy to see that the bijection given by the exponential map restricts to the bijection between LNDs on $\KK[X]$ of degree zero with respect to grading~\eqref{torgrad_eq} and $\GG_a$-actions on~$X$ normalized by the torus~$T$ in the automorphism group of~$X$. 

Any toric variety has a combinatorial description in terms of its fan, see~\cite[Section 3.1]{CLS}. For an affine toric variety $X$, the fan consists of all faces of the cone $\sigma$ dual to the cone~$\omega$ defined above, i.e. 
$
\sigma = \{v \in N_\QQ \mid \langle v, u\rangle \ge 0 \; \text{ for all } u \in \omega\},
$
where $\langle\,\cdot\,,\,\cdot\,\rangle$ is a natural pairing between $N_\QQ := N \otimes_\ZZ \QQ$ and $M_\QQ := M \otimes_\ZZ \QQ$. Let $\sigma(1) = \{\rho_i \mid 1 \le i \le m\}$ be the set of rays of a cone~$\sigma$. Denote by $p_i \in N$ the primitive vector on a ray~$\rho_i \in \sigma(1)$, and for any $1 \le i \le m$ set
\[
\mathfrak{R}_i = \{e \in M \mid \langle p_i, e\rangle = -1, \; \langle p_j, e\rangle \ge 0 \; \text{ for all } j \ne i, \; 1 \le j \le m\}.
\]
The elements of the set $\mathfrak{R} = \!\!\!\bigsqcup\limits_{1 \le i \le m} \!\!\!\mathfrak{R}_i$ are called the \emph{Demazure roots} of the toric variety~$X$. The Demazure roots of $X$ are in one-to-one correspondence with homogeneous LNDs on the algebra~$\KK[X]$ and hence with $\GG_a$-actions on~$X$ normalized by the acting torus~$T$, see~\cite[Theorem~2.7]{Li}. The derivation $\pa_e$ corresponding to a Demazure root $e \in \mathfrak{R}_i$ is given by the formula $\pa_e(\chi^u) = \langle p_i, u\rangle \chi^{u + e}$ for any $u \in M$. The Demazure root~$e$ is the degree of~$\pa_e$, and the action of the torus $T$ on $\GG_a$ by conjugation is the multiplication by $\chi^e$. 

\begin{example}
\label{A2_example}
Let $X = \AA^m$ and the torus $T = (\KK^\times)^m$ act on $\AA^m$ diagonally. Any character of~$T$ has the form $\chi^u(t) = t^u$, $u \in \ZZ^m$, and decomposition~\eqref{torgrad_eq} becomes 
\[
\KK[x_1, \ldots, x_m] = \bigoplus_{u \in \Zgezero^m} \KK x_1^{u_1} \ldots x_m^{u_m},
\]
i.e. $\chi^u(x) = x^u$, $S_X = \Zgezero^m$, $\omega = \QQ_{\ge 0}^m \subseteq M_\QQ$. Then the dual cone is $\sigma = \QQ_{\ge 0}^m \subseteq N_\QQ$, and $p_i$ are the elements of the standard basis. 

Any derivation on $\KK[x_1, \ldots, x_m]$ has the form $\pa = f_1 \frac{\pa}{\pa x_1} + \ldots + f_m \frac{\pa}{\pa x_m}$, $f_i \in \KK[x_1, \ldots, x_m]$. It is homogeneous with respect to the above grading if and only if it is equal to $\lambda_1 x_1 x^u \frac{\pa}{\pa x_1} + \ldots + \lambda_m x_m x^u \frac{\pa}{\pa x_m}$ or 
$\lambda_i x^u \frac{\pa}{\pa x_i}$ for some $u \in S_X$, $\lambda_i \in \KK$. 
It is easy to see that such a derivation is locally nilpotent if and only if it equals $\pa = \lambda x_1^{e_1} \ldots x_m^{e_m} \frac{\pa}{\pa x_i}$ with no~$x_i$ in the monomial. The degree of~$\pa$ equals $e = (e_1, \ldots, \underset{i}{-1}, \ldots, e_m)$ and is a Demazure root in accordance with the definition of~$\mathfrak{R}_i$ given above. 
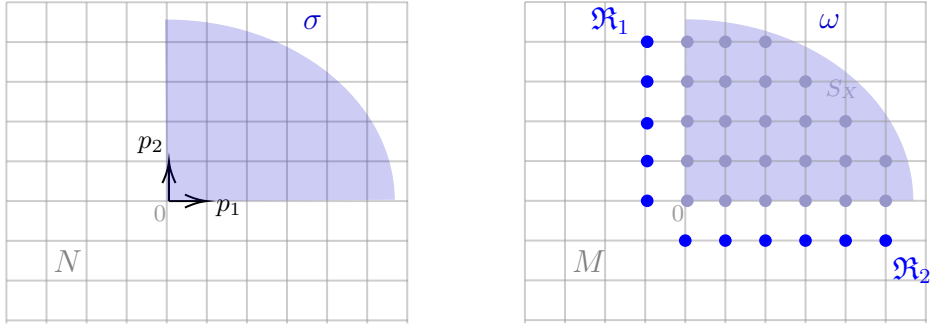
\begin{center}

\tikzset{every picture/.style={line width=0.75pt}} 

\begin{tikzpicture}[x=0.75pt,y=0.75pt,yscale=-1,xscale=1]

\draw  [draw opacity=0] (120.69,20.6) -- (321.08,20.6) -- (321.08,182.11) -- (120.69,182.11) -- cycle ; \draw  [color={rgb, 255:red, 155; green, 155; blue, 155 }  ,draw opacity=0.5 ] (120.69,20.6) -- (120.69,182.11)(140.69,20.6) -- (140.69,182.11)(160.69,20.6) -- (160.69,182.11)(180.69,20.6) -- (180.69,182.11)(200.69,20.6) -- (200.69,182.11)(220.69,20.6) -- (220.69,182.11)(240.69,20.6) -- (240.69,182.11)(260.69,20.6) -- (260.69,182.11)(280.69,20.6) -- (280.69,182.11)(300.69,20.6) -- (300.69,182.11)(320.69,20.6) -- (320.69,182.11) ; \draw  [color={rgb, 255:red, 155; green, 155; blue, 155 }  ,draw opacity=0.5 ] (120.69,20.6) -- (321.08,20.6)(120.69,40.6) -- (321.08,40.6)(120.69,60.6) -- (321.08,60.6)(120.69,80.6) -- (321.08,80.6)(120.69,100.6) -- (321.08,100.6)(120.69,120.6) -- (321.08,120.6)(120.69,140.6) -- (321.08,140.6)(120.69,160.6) -- (321.08,160.6)(120.69,180.6) -- (321.08,180.6) ; \draw  [color={rgb, 255:red, 155; green, 155; blue, 155 }  ,draw opacity=0.5 ]  ;
\draw    (201.8,120.6) -- (201.8,102.6) ;
\draw [shift={(201.8,100.6)}, rotate = 450] [color={rgb, 255:red, 0; green, 0; blue, 0 }  ][line width=0.75]    (10.93,-3.29) .. controls (6.95,-1.4) and (3.31,-0.3) .. (0,0) .. controls (3.31,0.3) and (6.95,1.4) .. (10.93,3.29)   ;
\draw    (201.8,120.6) -- (218.69,120.6) ;
\draw [shift={(220.69,120.6)}, rotate = 180] [color={rgb, 255:red, 0; green, 0; blue, 0 }  ][line width=0.75]    (10.93,-3.29) .. controls (6.95,-1.4) and (3.31,-0.3) .. (0,0) .. controls (3.31,0.3) and (6.95,1.4) .. (10.93,3.29)   ;
\draw  [draw opacity=0][fill={rgb, 255:red, 0; green, 0; blue, 200 }  ,fill opacity=0.2 ] (199.81,29.28) .. controls (200.11,29.28) and (200.4,29.28) .. (200.69,29.28) .. controls (263.27,29.28) and (314.06,69.82) .. (314.48,119.96) -- (200.69,120.6) -- cycle ;
\draw  [draw opacity=0][fill={rgb, 255:red, 0; green, 0; blue, 200 }  ,fill opacity=0.2 ] (459.51,29.15) .. controls (522.23,29.2) and (573.07,69.97) .. (573.19,120.29) -- (459.4,120.47) -- cycle ;
\draw  [draw opacity=0] (379.4,20.47) -- (579.79,20.47) -- (579.79,181.98) -- (379.4,181.98) -- cycle ; \draw  [color={rgb, 255:red, 155; green, 155; blue, 155 }  ,draw opacity=0.5 ] (379.4,20.47) -- (379.4,181.98)(399.4,20.47) -- (399.4,181.98)(419.4,20.47) -- (419.4,181.98)(439.4,20.47) -- (439.4,181.98)(459.4,20.47) -- (459.4,181.98)(479.4,20.47) -- (479.4,181.98)(499.4,20.47) -- (499.4,181.98)(519.4,20.47) -- (519.4,181.98)(539.4,20.47) -- (539.4,181.98)(559.4,20.47) -- (559.4,181.98)(579.4,20.47) -- (579.4,181.98) ; \draw  [color={rgb, 255:red, 155; green, 155; blue, 155 }  ,draw opacity=0.5 ] (379.4,20.47) -- (579.79,20.47)(379.4,40.47) -- (579.79,40.47)(379.4,60.47) -- (579.79,60.47)(379.4,80.47) -- (579.79,80.47)(379.4,100.47) -- (579.79,100.47)(379.4,120.47) -- (579.79,120.47)(379.4,140.47) -- (579.79,140.47)(379.4,160.47) -- (579.79,160.47)(379.4,180.47) -- (579.79,180.47) ; \draw  [color={rgb, 255:red, 155; green, 155; blue, 155 }  ,draw opacity=0.5 ]  ;
\draw  [draw opacity=0][fill={rgb, 255:red, 0; green, 0; blue, 255 }  ,fill opacity=1 ] (496.25,140.47) .. controls (496.25,138.73) and (497.66,137.32) .. (499.4,137.32) .. controls (501.14,137.32) and (502.55,138.73) .. (502.55,140.47) .. controls (502.55,142.21) and (501.14,143.62) .. (499.4,143.62) .. controls (497.66,143.62) and (496.25,142.21) .. (496.25,140.47) -- cycle ;
\draw  [draw opacity=0][fill={rgb, 255:red, 0; green, 0; blue, 255 }  ,fill opacity=1 ] (516.25,140.47) .. controls (516.25,138.73) and (517.66,137.32) .. (519.4,137.32) .. controls (521.14,137.32) and (522.55,138.73) .. (522.55,140.47) .. controls (522.55,142.21) and (521.14,143.62) .. (519.4,143.62) .. controls (517.66,143.62) and (516.25,142.21) .. (516.25,140.47) -- cycle ;
\draw  [draw opacity=0][fill={rgb, 255:red, 0; green, 0; blue, 255 }  ,fill opacity=1 ] (536.25,140.47) .. controls (536.25,138.73) and (537.66,137.32) .. (539.4,137.32) .. controls (541.14,137.32) and (542.55,138.73) .. (542.55,140.47) .. controls (542.55,142.21) and (541.14,143.62) .. (539.4,143.62) .. controls (537.66,143.62) and (536.25,142.21) .. (536.25,140.47) -- cycle ;
\draw  [draw opacity=0][fill={rgb, 255:red, 0; green, 0; blue, 255 }  ,fill opacity=1 ] (556.25,140.47) .. controls (556.25,138.73) and (557.66,137.32) .. (559.4,137.32) .. controls (561.14,137.32) and (562.55,138.73) .. (562.55,140.47) .. controls (562.55,142.21) and (561.14,143.62) .. (559.4,143.62) .. controls (557.66,143.62) and (556.25,142.21) .. (556.25,140.47) -- cycle ;
\draw  [draw opacity=0][fill={rgb, 255:red, 0; green, 0; blue, 255 }  ,fill opacity=1 ] (476.25,140.47) .. controls (476.25,138.73) and (477.66,137.32) .. (479.4,137.32) .. controls (481.14,137.32) and (482.55,138.73) .. (482.55,140.47) .. controls (482.55,142.21) and (481.14,143.62) .. (479.4,143.62) .. controls (477.66,143.62) and (476.25,142.21) .. (476.25,140.47) -- cycle ;
\draw  [draw opacity=0][fill={rgb, 255:red, 0; green, 0; blue, 255 }  ,fill opacity=1 ] (437.25,100.47) .. controls (437.25,98.73) and (438.66,97.32) .. (440.4,97.32) .. controls (442.14,97.32) and (443.55,98.73) .. (443.55,100.47) .. controls (443.55,102.21) and (442.14,103.62) .. (440.4,103.62) .. controls (438.66,103.62) and (437.25,102.21) .. (437.25,100.47) -- cycle ;
\draw  [draw opacity=0][fill={rgb, 255:red, 0; green, 0; blue, 255 }  ,fill opacity=1 ] (437.25,60.47) .. controls (437.25,58.73) and (438.66,57.32) .. (440.4,57.32) .. controls (442.14,57.32) and (443.55,58.73) .. (443.55,60.47) .. controls (443.55,62.21) and (442.14,63.62) .. (440.4,63.62) .. controls (438.66,63.62) and (437.25,62.21) .. (437.25,60.47) -- cycle ;
\draw  [draw opacity=0][fill={rgb, 255:red, 0; green, 0; blue, 255 }  ,fill opacity=1 ] (437.25,40.47) .. controls (437.25,38.73) and (438.66,37.32) .. (440.4,37.32) .. controls (442.14,37.32) and (443.55,38.73) .. (443.55,40.47) .. controls (443.55,42.21) and (442.14,43.62) .. (440.4,43.62) .. controls (438.66,43.62) and (437.25,42.21) .. (437.25,40.47) -- cycle ;
\draw  [draw opacity=0][fill={rgb, 255:red, 0; green, 0; blue, 255 }  ,fill opacity=1 ] (456.25,140.47) .. controls (456.25,138.73) and (457.66,137.32) .. (459.4,137.32) .. controls (461.14,137.32) and (462.55,138.73) .. (462.55,140.47) .. controls (462.55,142.21) and (461.14,143.62) .. (459.4,143.62) .. controls (457.66,143.62) and (456.25,142.21) .. (456.25,140.47) -- cycle ;
\draw  [draw opacity=0][fill={rgb, 255:red, 0; green, 0; blue, 255 }  ,fill opacity=1 ] (437.25,120.47) .. controls (437.25,118.73) and (438.66,117.32) .. (440.4,117.32) .. controls (442.14,117.32) and (443.55,118.73) .. (443.55,120.47) .. controls (443.55,122.21) and (442.14,123.62) .. (440.4,123.62) .. controls (438.66,123.62) and (437.25,122.21) .. (437.25,120.47) -- cycle ;
\draw  [draw opacity=0][fill={rgb, 255:red, 0; green, 0; blue, 255 }  ,fill opacity=1 ] (437.25,81.47) .. controls (437.25,79.73) and (438.66,78.32) .. (440.4,78.32) .. controls (442.14,78.32) and (443.55,79.73) .. (443.55,81.47) .. controls (443.55,83.21) and (442.14,84.62) .. (440.4,84.62) .. controls (438.66,84.62) and (437.25,83.21) .. (437.25,81.47) -- cycle ;
\draw  [draw opacity=0][fill={rgb, 255:red, 150; green, 150; blue, 200 }  ,fill opacity=1 ] (476.25,100.47) .. controls (476.25,98.73) and (477.66,97.32) .. (479.4,97.32) .. controls (481.14,97.32) and (482.55,98.73) .. (482.55,100.47) .. controls (482.55,102.21) and (481.14,103.62) .. (479.4,103.62) .. controls (477.66,103.62) and (476.25,102.21) .. (476.25,100.47) -- cycle ;
\draw  [draw opacity=0][fill={rgb, 255:red, 150; green, 150; blue, 200 }  ,fill opacity=1 ] (496.25,100.47) .. controls (496.25,98.73) and (497.66,97.32) .. (499.4,97.32) .. controls (501.14,97.32) and (502.55,98.73) .. (502.55,100.47) .. controls (502.55,102.21) and (501.14,103.62) .. (499.4,103.62) .. controls (497.66,103.62) and (496.25,102.21) .. (496.25,100.47) -- cycle ;
\draw  [draw opacity=0][fill={rgb, 255:red, 150; green, 150; blue, 200 }  ,fill opacity=1 ] (457.25,120.47) .. controls (457.25,118.73) and (458.66,117.32) .. (460.4,117.32) .. controls (462.14,117.32) and (463.55,118.73) .. (463.55,120.47) .. controls (463.55,122.21) and (462.14,123.62) .. (460.4,123.62) .. controls (458.66,123.62) and (457.25,122.21) .. (457.25,120.47) -- cycle ;
\draw  [draw opacity=0][fill={rgb, 255:red, 150; green, 150; blue, 200 }  ,fill opacity=1 ] (476.25,120.47) .. controls (476.25,118.73) and (477.66,117.32) .. (479.4,117.32) .. controls (481.14,117.32) and (482.55,118.73) .. (482.55,120.47) .. controls (482.55,122.21) and (481.14,123.62) .. (479.4,123.62) .. controls (477.66,123.62) and (476.25,122.21) .. (476.25,120.47) -- cycle ;
\draw  [draw opacity=0][fill={rgb, 255:red, 150; green, 150; blue, 200 }  ,fill opacity=1 ] (496.25,80.47) .. controls (496.25,78.73) and (497.66,77.32) .. (499.4,77.32) .. controls (501.14,77.32) and (502.55,78.73) .. (502.55,80.47) .. controls (502.55,82.21) and (501.14,83.62) .. (499.4,83.62) .. controls (497.66,83.62) and (496.25,82.21) .. (496.25,80.47) -- cycle ;
\draw  [draw opacity=0][fill={rgb, 255:red, 150; green, 150; blue, 200 }  ,fill opacity=1 ] (476.25,80.47) .. controls (476.25,78.73) and (477.66,77.32) .. (479.4,77.32) .. controls (481.14,77.32) and (482.55,78.73) .. (482.55,80.47) .. controls (482.55,82.21) and (481.14,83.62) .. (479.4,83.62) .. controls (477.66,83.62) and (476.25,82.21) .. (476.25,80.47) -- cycle ;
\draw  [draw opacity=0][fill={rgb, 255:red, 150; green, 150; blue, 200 }  ,fill opacity=1 ] (457.25,100.47) .. controls (457.25,98.73) and (458.66,97.32) .. (460.4,97.32) .. controls (462.14,97.32) and (463.55,98.73) .. (463.55,100.47) .. controls (463.55,102.21) and (462.14,103.62) .. (460.4,103.62) .. controls (458.66,103.62) and (457.25,102.21) .. (457.25,100.47) -- cycle ;
\draw  [draw opacity=0][fill={rgb, 255:red, 150; green, 150; blue, 200 }  ,fill opacity=1 ] (457.25,80.47) .. controls (457.25,78.73) and (458.66,77.32) .. (460.4,77.32) .. controls (462.14,77.32) and (463.55,78.73) .. (463.55,80.47) .. controls (463.55,82.21) and (462.14,83.62) .. (460.4,83.62) .. controls (458.66,83.62) and (457.25,82.21) .. (457.25,80.47) -- cycle ;
\draw  [draw opacity=0][fill={rgb, 255:red, 150; green, 150; blue, 200 }  ,fill opacity=1 ] (457.25,60.47) .. controls (457.25,58.73) and (458.66,57.32) .. (460.4,57.32) .. controls (462.14,57.32) and (463.55,58.73) .. (463.55,60.47) .. controls (463.55,62.21) and (462.14,63.62) .. (460.4,63.62) .. controls (458.66,63.62) and (457.25,62.21) .. (457.25,60.47) -- cycle ;
\draw  [draw opacity=0][fill={rgb, 255:red, 150; green, 150; blue, 200 }  ,fill opacity=1 ] (457.25,40.47) .. controls (457.25,38.73) and (458.66,37.32) .. (460.4,37.32) .. controls (462.14,37.32) and (463.55,38.73) .. (463.55,40.47) .. controls (463.55,42.21) and (462.14,43.62) .. (460.4,43.62) .. controls (458.66,43.62) and (457.25,42.21) .. (457.25,40.47) -- cycle ;
\draw  [draw opacity=0][fill={rgb, 255:red, 150; green, 150; blue, 200 }  ,fill opacity=1 ] (476.25,60.47) .. controls (476.25,58.73) and (477.66,57.32) .. (479.4,57.32) .. controls (481.14,57.32) and (482.55,58.73) .. (482.55,60.47) .. controls (482.55,62.21) and (481.14,63.62) .. (479.4,63.62) .. controls (477.66,63.62) and (476.25,62.21) .. (476.25,60.47) -- cycle ;
\draw  [draw opacity=0][fill={rgb, 255:red, 150; green, 150; blue, 200 }  ,fill opacity=1 ] (476.25,40.47) .. controls (476.25,38.73) and (477.66,37.32) .. (479.4,37.32) .. controls (481.14,37.32) and (482.55,38.73) .. (482.55,40.47) .. controls (482.55,42.21) and (481.14,43.62) .. (479.4,43.62) .. controls (477.66,43.62) and (476.25,42.21) .. (476.25,40.47) -- cycle ;
\draw  [draw opacity=0][fill={rgb, 255:red, 150; green, 150; blue, 200 }  ,fill opacity=1 ] (496.25,60.47) .. controls (496.25,58.73) and (497.66,57.32) .. (499.4,57.32) .. controls (501.14,57.32) and (502.55,58.73) .. (502.55,60.47) .. controls (502.55,62.21) and (501.14,63.62) .. (499.4,63.62) .. controls (497.66,63.62) and (496.25,62.21) .. (496.25,60.47) -- cycle ;
\draw  [draw opacity=0][fill={rgb, 255:red, 150; green, 150; blue, 200 }  ,fill opacity=1 ] (496.25,40.47) .. controls (496.25,38.73) and (497.66,37.32) .. (499.4,37.32) .. controls (501.14,37.32) and (502.55,38.73) .. (502.55,40.47) .. controls (502.55,42.21) and (501.14,43.62) .. (499.4,43.62) .. controls (497.66,43.62) and (496.25,42.21) .. (496.25,40.47) -- cycle ;
\draw  [draw opacity=0][fill={rgb, 255:red, 150; green, 150; blue, 200 }  ,fill opacity=1 ] (516.25,60.47) .. controls (516.25,58.73) and (517.66,57.32) .. (519.4,57.32) .. controls (521.14,57.32) and (522.55,58.73) .. (522.55,60.47) .. controls (522.55,62.21) and (521.14,63.62) .. (519.4,63.62) .. controls (517.66,63.62) and (516.25,62.21) .. (516.25,60.47) -- cycle ;
\draw  [draw opacity=0][fill={rgb, 255:red, 150; green, 150; blue, 200 }  ,fill opacity=1 ] (516.25,80.47) .. controls (516.25,78.73) and (517.66,77.32) .. (519.4,77.32) .. controls (521.14,77.32) and (522.55,78.73) .. (522.55,80.47) .. controls (522.55,82.21) and (521.14,83.62) .. (519.4,83.62) .. controls (517.66,83.62) and (516.25,82.21) .. (516.25,80.47) -- cycle ;
\draw  [draw opacity=0][fill={rgb, 255:red, 150; green, 150; blue, 200 }  ,fill opacity=1 ] (516.25,100.47) .. controls (516.25,98.73) and (517.66,97.32) .. (519.4,97.32) .. controls (521.14,97.32) and (522.55,98.73) .. (522.55,100.47) .. controls (522.55,102.21) and (521.14,103.62) .. (519.4,103.62) .. controls (517.66,103.62) and (516.25,102.21) .. (516.25,100.47) -- cycle ;
\draw  [draw opacity=0][fill={rgb, 255:red, 150; green, 150; blue, 200 }  ,fill opacity=1 ] (536.25,80.47) .. controls (536.25,78.73) and (537.66,77.32) .. (539.4,77.32) .. controls (541.14,77.32) and (542.55,78.73) .. (542.55,80.47) .. controls (542.55,82.21) and (541.14,83.62) .. (539.4,83.62) .. controls (537.66,83.62) and (536.25,82.21) .. (536.25,80.47) -- cycle ;
\draw  [draw opacity=0][fill={rgb, 255:red, 150; green, 150; blue, 200 }  ,fill opacity=1 ] (556.25,100.47) .. controls (556.25,98.73) and (557.66,97.32) .. (559.4,97.32) .. controls (561.14,97.32) and (562.55,98.73) .. (562.55,100.47) .. controls (562.55,102.21) and (561.14,103.62) .. (559.4,103.62) .. controls (557.66,103.62) and (556.25,102.21) .. (556.25,100.47) -- cycle ;
\draw  [draw opacity=0][fill={rgb, 255:red, 150; green, 150; blue, 200 }  ,fill opacity=1 ] (536.25,100.47) .. controls (536.25,98.73) and (537.66,97.32) .. (539.4,97.32) .. controls (541.14,97.32) and (542.55,98.73) .. (542.55,100.47) .. controls (542.55,102.21) and (541.14,103.62) .. (539.4,103.62) .. controls (537.66,103.62) and (536.25,102.21) .. (536.25,100.47) -- cycle ;
\draw  [draw opacity=0][fill={rgb, 255:red, 150; green, 150; blue, 200 }  ,fill opacity=1 ] (556.25,120.47) .. controls (556.25,118.73) and (557.66,117.32) .. (559.4,117.32) .. controls (561.14,117.32) and (562.55,118.73) .. (562.55,120.47) .. controls (562.55,122.21) and (561.14,123.62) .. (559.4,123.62) .. controls (557.66,123.62) and (556.25,122.21) .. (556.25,120.47) -- cycle ;
\draw  [draw opacity=0][fill={rgb, 255:red, 150; green, 150; blue, 200 }  ,fill opacity=1 ] (536.25,120.47) .. controls (536.25,118.73) and (537.66,117.32) .. (539.4,117.32) .. controls (541.14,117.32) and (542.55,118.73) .. (542.55,120.47) .. controls (542.55,122.21) and (541.14,123.62) .. (539.4,123.62) .. controls (537.66,123.62) and (536.25,122.21) .. (536.25,120.47) -- cycle ;
\draw  [draw opacity=0][fill={rgb, 255:red, 150; green, 150; blue, 200 }  ,fill opacity=1 ] (516.25,120.47) .. controls (516.25,118.73) and (517.66,117.32) .. (519.4,117.32) .. controls (521.14,117.32) and (522.55,118.73) .. (522.55,120.47) .. controls (522.55,122.21) and (521.14,123.62) .. (519.4,123.62) .. controls (517.66,123.62) and (516.25,122.21) .. (516.25,120.47) -- cycle ;
\draw  [draw opacity=0][fill={rgb, 255:red, 150; green, 150; blue, 200 }  ,fill opacity=1 ] (496.25,120.47) .. controls (496.25,118.73) and (497.66,117.32) .. (499.4,117.32) .. controls (501.14,117.32) and (502.55,118.73) .. (502.55,120.47) .. controls (502.55,122.21) and (501.14,123.62) .. (499.4,123.62) .. controls (497.66,123.62) and (496.25,122.21) .. (496.25,120.47) -- cycle ;

\draw (185,86.7) node [anchor=north west][inner sep=0.75pt]  [font=\footnotesize] [align=left] {$\displaystyle p_{2}$};
\draw (224,118.5) node [anchor=north west][inner sep=0.75pt]   [align=left] {{\footnotesize $\displaystyle p_{1}$}};
\draw (142.69,144) node [anchor=north west][inner sep=0.75pt]  [color={rgb, 255:red, 128; green, 128; blue, 128 }  ,opacity=1 ]  {$N$};
\draw (401.4,143.87) node [anchor=north west][inner sep=0.75pt]  [color={rgb, 255:red, 128; green, 128; blue, 128 }  ,opacity=1 ]  {$M$};
\draw (561.4,147.02) node [anchor=north west][inner sep=0.75pt]  [color={rgb, 255:red, 0; green, 0; blue, 255 }  ,opacity=1 ]  {$\mathfrak{R}_{2}$};
\draw (411,21.47) node [anchor=north west][inner sep=0.75pt]  [color={rgb, 255:red, 0; green, 0; blue, 255 }  ,opacity=1 ]  {$\mathfrak{R}_{1}$};
\draw (267.3,25.7) node [anchor=north west][inner sep=0.75pt]  [color={rgb, 255:red, 208; green, 2; blue, 27 }  ,opacity=1 ]  {$\textcolor[rgb]{0,0,0.78}{\sigma }$};
\draw (524.9,25.37) node [anchor=north west][inner sep=0.75pt]  [color={rgb, 255:red, 0; green, 0; blue, 200 }  ,opacity=1 ]  {$\textcolor[rgb]{0,0,0.78}{\omega }$};
\draw (192.8,121.2) node [anchor=north west][inner sep=0.75pt]  [font=\scriptsize,color={rgb, 255:red, 155; green, 155; blue, 155 }  ,opacity=1 ]  {$0$};
\draw (451.2,121.6) node [anchor=north west][inner sep=0.75pt]  [font=\scriptsize,color={rgb, 255:red, 155; green, 155; blue, 155 }  ,opacity=1 ]  {$0$};
\draw (528,57.9) node [anchor=north west][inner sep=0.75pt]  [font=\scriptsize,color={rgb, 255:red, 150; green, 150; blue, 200 }  ,opacity=1 ]  {$S_{X}$};

\end{tikzpicture}

\end{center}
\end{example}

\begin{example}
\label{barX_example}
Let $X = \AA^m \times (\KK^\times)^{\widetilde m}$. In this case, we have the semigroup $S_X = \Zgezero^m \times \ZZ^{\widetilde m}$ in $M = \ZZ^{m + \widetilde m}$ spanning the cone $\omega$, and the dual cone $\sigma$ is spanned by the first $m$ vectors $p_1, \ldots, p_m$ of the standard basis. Then for $1 \le i \le m$, the set $\mathfrak{R}_i$ consists of vectors \[e = (e_1, \ldots, \underset{i}{-1}, \ldots, e_m, \tilde e_{m+1}, \ldots, \tilde e_{m+\widetilde m}), \; e_j \in \Zgezero, \; \tilde e_j \in \ZZ.\] 
\end{example}

\begin{example}
\label{cone_1example}
Let $X$ be an affine toric surface given by the cone~$\sigma$ with $p_1 = (0, 1)$ and $p_2 = (2, -1)$. Then the semigroup $S_X = M \cap \,\omega$ is generated by vectors $(1,0)$, $(1,1)$ and $(1,2)$. Denoting $a=\chi^{(1,0)}$, $b=\chi^{(1,1)}$ and $c=\chi^{(1,2)}$, we obtain $\KK[X] = \KK[a, b, c]$ with relation $ac = b^2$, i.e. the surface $X$ is isomorphic to the quadratic cone $\{ac =  b^2\} \subset \AA^3$. 
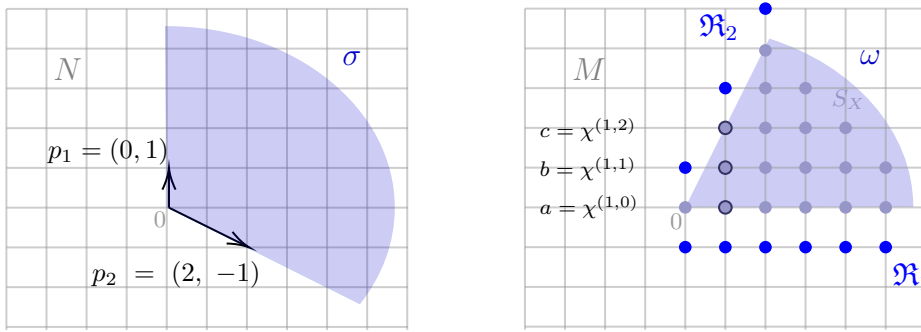
\begin{center}

\tikzset{every picture/.style={line width=0.75pt}} 

\begin{tikzpicture}[x=0.75pt,y=0.75pt,yscale=-1,xscale=1]

\draw  [draw opacity=0] (120.69,20.6) -- (321.08,20.6) -- (321.08,182.11) -- (120.69,182.11) -- cycle ; \draw  [color={rgb, 255:red, 155; green, 155; blue, 155 }  ,draw opacity=0.5 ] (120.69,20.6) -- (120.69,182.11)(140.69,20.6) -- (140.69,182.11)(160.69,20.6) -- (160.69,182.11)(180.69,20.6) -- (180.69,182.11)(200.69,20.6) -- (200.69,182.11)(220.69,20.6) -- (220.69,182.11)(240.69,20.6) -- (240.69,182.11)(260.69,20.6) -- (260.69,182.11)(280.69,20.6) -- (280.69,182.11)(300.69,20.6) -- (300.69,182.11)(320.69,20.6) -- (320.69,182.11) ; \draw  [color={rgb, 255:red, 155; green, 155; blue, 155 }  ,draw opacity=0.5 ] (120.69,20.6) -- (321.08,20.6)(120.69,40.6) -- (321.08,40.6)(120.69,60.6) -- (321.08,60.6)(120.69,80.6) -- (321.08,80.6)(120.69,100.6) -- (321.08,100.6)(120.69,120.6) -- (321.08,120.6)(120.69,140.6) -- (321.08,140.6)(120.69,160.6) -- (321.08,160.6)(120.69,180.6) -- (321.08,180.6) ; \draw  [color={rgb, 255:red, 155; green, 155; blue, 155 }  ,draw opacity=0.5 ]  ;
\draw    (201.8,120.6) -- (201.8,102.6) ;
\draw [shift={(201.8,100.6)}, rotate = 450] [color={rgb, 255:red, 0; green, 0; blue, 0 }  ][line width=0.75]    (10.93,-3.29) .. controls (6.95,-1.4) and (3.31,-0.3) .. (0,0) .. controls (3.31,0.3) and (6.95,1.4) .. (10.93,3.29)   ;
\draw    (201.8,120.6) -- (240.01,139.71) ;
\draw [shift={(241.8,140.6)}, rotate = 206.57] [color={rgb, 255:red, 0; green, 0; blue, 0 }  ][line width=0.75]    (10.93,-3.29) .. controls (6.95,-1.4) and (3.31,-0.3) .. (0,0) .. controls (3.31,0.3) and (6.95,1.4) .. (10.93,3.29)   ;
\draw  [draw opacity=0][fill={rgb, 255:red, 0; green, 0; blue, 200 }  ,fill opacity=0.2 ] (199.81,29.28) .. controls (200.11,29.28) and (200.4,29.28) .. (200.69,29.28) .. controls (263.54,29.28) and (314.48,70.17) .. (314.48,120.6) .. controls (314.48,138.47) and (308.09,155.14) .. (297.03,169.22) -- (200.69,120.6) -- cycle ;
\draw  [draw opacity=0][fill={rgb, 255:red, 0; green, 0; blue, 200 }  ,fill opacity=0.2 ] (501.32,35.55) .. controls (543.35,48.92) and (573.1,81.81) .. (573.19,120.29) -- (459.4,120.47) -- cycle ;
\draw  [draw opacity=0] (379.4,20.47) -- (579.79,20.47) -- (579.79,181.98) -- (379.4,181.98) -- cycle ; \draw  [color={rgb, 255:red, 155; green, 155; blue, 155 }  ,draw opacity=0.5 ] (379.4,20.47) -- (379.4,181.98)(399.4,20.47) -- (399.4,181.98)(419.4,20.47) -- (419.4,181.98)(439.4,20.47) -- (439.4,181.98)(459.4,20.47) -- (459.4,181.98)(479.4,20.47) -- (479.4,181.98)(499.4,20.47) -- (499.4,181.98)(519.4,20.47) -- (519.4,181.98)(539.4,20.47) -- (539.4,181.98)(559.4,20.47) -- (559.4,181.98)(579.4,20.47) -- (579.4,181.98) ; \draw  [color={rgb, 255:red, 155; green, 155; blue, 155 }  ,draw opacity=0.5 ] (379.4,20.47) -- (579.79,20.47)(379.4,40.47) -- (579.79,40.47)(379.4,60.47) -- (579.79,60.47)(379.4,80.47) -- (579.79,80.47)(379.4,100.47) -- (579.79,100.47)(379.4,120.47) -- (579.79,120.47)(379.4,140.47) -- (579.79,140.47)(379.4,160.47) -- (579.79,160.47)(379.4,180.47) -- (579.79,180.47) ; \draw  [color={rgb, 255:red, 155; green, 155; blue, 155 }  ,draw opacity=0.5 ]  ;
\draw  [draw opacity=0][fill={rgb, 255:red, 0; green, 0; blue, 255 }  ,fill opacity=1 ] (496.25,140.47) .. controls (496.25,138.73) and (497.66,137.32) .. (499.4,137.32) .. controls (501.14,137.32) and (502.55,138.73) .. (502.55,140.47) .. controls (502.55,142.21) and (501.14,143.62) .. (499.4,143.62) .. controls (497.66,143.62) and (496.25,142.21) .. (496.25,140.47) -- cycle ;
\draw  [draw opacity=0][fill={rgb, 255:red, 0; green, 0; blue, 255 }  ,fill opacity=1 ] (516.25,140.47) .. controls (516.25,138.73) and (517.66,137.32) .. (519.4,137.32) .. controls (521.14,137.32) and (522.55,138.73) .. (522.55,140.47) .. controls (522.55,142.21) and (521.14,143.62) .. (519.4,143.62) .. controls (517.66,143.62) and (516.25,142.21) .. (516.25,140.47) -- cycle ;
\draw  [draw opacity=0][fill={rgb, 255:red, 0; green, 0; blue, 255 }  ,fill opacity=1 ] (536.25,140.47) .. controls (536.25,138.73) and (537.66,137.32) .. (539.4,137.32) .. controls (541.14,137.32) and (542.55,138.73) .. (542.55,140.47) .. controls (542.55,142.21) and (541.14,143.62) .. (539.4,143.62) .. controls (537.66,143.62) and (536.25,142.21) .. (536.25,140.47) -- cycle ;
\draw  [draw opacity=0][fill={rgb, 255:red, 0; green, 0; blue, 255 }  ,fill opacity=1 ] (556.25,140.47) .. controls (556.25,138.73) and (557.66,137.32) .. (559.4,137.32) .. controls (561.14,137.32) and (562.55,138.73) .. (562.55,140.47) .. controls (562.55,142.21) and (561.14,143.62) .. (559.4,143.62) .. controls (557.66,143.62) and (556.25,142.21) .. (556.25,140.47) -- cycle ;
\draw  [draw opacity=0][fill={rgb, 255:red, 0; green, 0; blue, 255 }  ,fill opacity=1 ] (476.25,140.47) .. controls (476.25,138.73) and (477.66,137.32) .. (479.4,137.32) .. controls (481.14,137.32) and (482.55,138.73) .. (482.55,140.47) .. controls (482.55,142.21) and (481.14,143.62) .. (479.4,143.62) .. controls (477.66,143.62) and (476.25,142.21) .. (476.25,140.47) -- cycle ;
\draw  [draw opacity=0][fill={rgb, 255:red, 0; green, 0; blue, 255 }  ,fill opacity=1 ] (456.25,100.47) .. controls (456.25,98.73) and (457.66,97.32) .. (459.4,97.32) .. controls (461.14,97.32) and (462.55,98.73) .. (462.55,100.47) .. controls (462.55,102.21) and (461.14,103.62) .. (459.4,103.62) .. controls (457.66,103.62) and (456.25,102.21) .. (456.25,100.47) -- cycle ;
\draw  [draw opacity=0][fill={rgb, 255:red, 0; green, 0; blue, 255 }  ,fill opacity=1 ] (476.25,60.47) .. controls (476.25,58.73) and (477.66,57.32) .. (479.4,57.32) .. controls (481.14,57.32) and (482.55,58.73) .. (482.55,60.47) .. controls (482.55,62.21) and (481.14,63.62) .. (479.4,63.62) .. controls (477.66,63.62) and (476.25,62.21) .. (476.25,60.47) -- cycle ;
\draw  [draw opacity=0][fill={rgb, 255:red, 0; green, 0; blue, 255 }  ,fill opacity=1 ] (496.25,20.47) .. controls (496.25,18.73) and (497.66,17.32) .. (499.4,17.32) .. controls (501.14,17.32) and (502.55,18.73) .. (502.55,20.47) .. controls (502.55,22.21) and (501.14,23.62) .. (499.4,23.62) .. controls (497.66,23.62) and (496.25,22.21) .. (496.25,20.47) -- cycle ;
\draw  [draw opacity=0][fill={rgb, 255:red, 0; green, 0; blue, 255 }  ,fill opacity=1 ] (456.25,140.47) .. controls (456.25,138.73) and (457.66,137.32) .. (459.4,137.32) .. controls (461.14,137.32) and (462.55,138.73) .. (462.55,140.47) .. controls (462.55,142.21) and (461.14,143.62) .. (459.4,143.62) .. controls (457.66,143.62) and (456.25,142.21) .. (456.25,140.47) -- cycle ;
\draw  [draw opacity=0][fill={rgb, 255:red, 150; green, 150; blue, 200 }  ,fill opacity=1 ] (496.25,120.47) .. controls (496.25,118.73) and (497.66,117.32) .. (499.4,117.32) .. controls (501.14,117.32) and (502.55,118.73) .. (502.55,120.47) .. controls (502.55,122.21) and (501.14,123.62) .. (499.4,123.62) .. controls (497.66,123.62) and (496.25,122.21) .. (496.25,120.47) -- cycle ;
\draw  [draw opacity=0][fill={rgb, 255:red, 150; green, 150; blue, 200 }  ,fill opacity=1 ] (516.25,120.47) .. controls (516.25,118.73) and (517.66,117.32) .. (519.4,117.32) .. controls (521.14,117.32) and (522.55,118.73) .. (522.55,120.47) .. controls (522.55,122.21) and (521.14,123.62) .. (519.4,123.62) .. controls (517.66,123.62) and (516.25,122.21) .. (516.25,120.47) -- cycle ;
\draw  [draw opacity=0][fill={rgb, 255:red, 150; green, 150; blue, 200 }  ,fill opacity=1 ] (516.25,80.47) .. controls (516.25,78.73) and (517.66,77.32) .. (519.4,77.32) .. controls (521.14,77.32) and (522.55,78.73) .. (522.55,80.47) .. controls (522.55,82.21) and (521.14,83.62) .. (519.4,83.62) .. controls (517.66,83.62) and (516.25,82.21) .. (516.25,80.47) -- cycle ;
\draw  [draw opacity=0][fill={rgb, 255:red, 150; green, 150; blue, 200 }  ,fill opacity=1 ] (516.25,100.47) .. controls (516.25,98.73) and (517.66,97.32) .. (519.4,97.32) .. controls (521.14,97.32) and (522.55,98.73) .. (522.55,100.47) .. controls (522.55,102.21) and (521.14,103.62) .. (519.4,103.62) .. controls (517.66,103.62) and (516.25,102.21) .. (516.25,100.47) -- cycle ;
\draw  [draw opacity=0][fill={rgb, 255:red, 150; green, 150; blue, 200 }  ,fill opacity=1 ] (496.25,100.47) .. controls (496.25,98.73) and (497.66,97.32) .. (499.4,97.32) .. controls (501.14,97.32) and (502.55,98.73) .. (502.55,100.47) .. controls (502.55,102.21) and (501.14,103.62) .. (499.4,103.62) .. controls (497.66,103.62) and (496.25,102.21) .. (496.25,100.47) -- cycle ;
\draw  [draw opacity=0][fill={rgb, 255:red, 150; green, 150; blue, 200 }  ,fill opacity=1 ] (496.25,80.47) .. controls (496.25,78.73) and (497.66,77.32) .. (499.4,77.32) .. controls (501.14,77.32) and (502.55,78.73) .. (502.55,80.47) .. controls (502.55,82.21) and (501.14,83.62) .. (499.4,83.62) .. controls (497.66,83.62) and (496.25,82.21) .. (496.25,80.47) -- cycle ;
\draw  [draw opacity=0][fill={rgb, 255:red, 150; green, 150; blue, 200 }  ,fill opacity=1 ] (516.25,60.47) .. controls (516.25,58.73) and (517.66,57.32) .. (519.4,57.32) .. controls (521.14,57.32) and (522.55,58.73) .. (522.55,60.47) .. controls (522.55,62.21) and (521.14,63.62) .. (519.4,63.62) .. controls (517.66,63.62) and (516.25,62.21) .. (516.25,60.47) -- cycle ;
\draw  [draw opacity=0][fill={rgb, 255:red, 150; green, 150; blue, 200 }  ,fill opacity=1 ] (496.25,60.47) .. controls (496.25,58.73) and (497.66,57.32) .. (499.4,57.32) .. controls (501.14,57.32) and (502.55,58.73) .. (502.55,60.47) .. controls (502.55,62.21) and (501.14,63.62) .. (499.4,63.62) .. controls (497.66,63.62) and (496.25,62.21) .. (496.25,60.47) -- cycle ;
\draw  [draw opacity=0][fill={rgb, 255:red, 150; green, 150; blue, 200 }  ,fill opacity=1 ] (496.25,41.47) .. controls (496.25,39.73) and (497.66,38.32) .. (499.4,38.32) .. controls (501.14,38.32) and (502.55,39.73) .. (502.55,41.47) .. controls (502.55,43.21) and (501.14,44.62) .. (499.4,44.62) .. controls (497.66,44.62) and (496.25,43.21) .. (496.25,41.47) -- cycle ;
\draw  [draw opacity=0][fill={rgb, 255:red, 150; green, 150; blue, 200 }  ,fill opacity=1 ] (456.25,120.47) .. controls (456.25,118.73) and (457.66,117.32) .. (459.4,117.32) .. controls (461.14,117.32) and (462.55,118.73) .. (462.55,120.47) .. controls (462.55,122.21) and (461.14,123.62) .. (459.4,123.62) .. controls (457.66,123.62) and (456.25,122.21) .. (456.25,120.47) -- cycle ;
\draw  [draw opacity=0][fill={rgb, 255:red, 150; green, 150; blue, 200 }  ,fill opacity=1 ] (536.25,80.47) .. controls (536.25,78.73) and (537.66,77.32) .. (539.4,77.32) .. controls (541.14,77.32) and (542.55,78.73) .. (542.55,80.47) .. controls (542.55,82.21) and (541.14,83.62) .. (539.4,83.62) .. controls (537.66,83.62) and (536.25,82.21) .. (536.25,80.47) -- cycle ;
\draw  [draw opacity=0][fill={rgb, 255:red, 150; green, 150; blue, 200 }  ,fill opacity=1 ] (536.25,100.47) .. controls (536.25,98.73) and (537.66,97.32) .. (539.4,97.32) .. controls (541.14,97.32) and (542.55,98.73) .. (542.55,100.47) .. controls (542.55,102.21) and (541.14,103.62) .. (539.4,103.62) .. controls (537.66,103.62) and (536.25,102.21) .. (536.25,100.47) -- cycle ;
\draw  [draw opacity=0][fill={rgb, 255:red, 150; green, 150; blue, 200 }  ,fill opacity=1 ] (536.25,120.47) .. controls (536.25,118.73) and (537.66,117.32) .. (539.4,117.32) .. controls (541.14,117.32) and (542.55,118.73) .. (542.55,120.47) .. controls (542.55,122.21) and (541.14,123.62) .. (539.4,123.62) .. controls (537.66,123.62) and (536.25,122.21) .. (536.25,120.47) -- cycle ;
\draw  [draw opacity=0][fill={rgb, 255:red, 150; green, 150; blue, 200 }  ,fill opacity=1 ] (556.25,100.47) .. controls (556.25,98.73) and (557.66,97.32) .. (559.4,97.32) .. controls (561.14,97.32) and (562.55,98.73) .. (562.55,100.47) .. controls (562.55,102.21) and (561.14,103.62) .. (559.4,103.62) .. controls (557.66,103.62) and (556.25,102.21) .. (556.25,100.47) -- cycle ;
\draw  [draw opacity=0][fill={rgb, 255:red, 150; green, 150; blue, 200 }  ,fill opacity=1 ] (556.25,120.47) .. controls (556.25,118.73) and (557.66,117.32) .. (559.4,117.32) .. controls (561.14,117.32) and (562.55,118.73) .. (562.55,120.47) .. controls (562.55,122.21) and (561.14,123.62) .. (559.4,123.62) .. controls (557.66,123.62) and (556.25,122.21) .. (556.25,120.47) -- cycle ;
\draw  [color={rgb, 255:red, 50; green, 50; blue, 100 }  ,draw opacity=1 ][fill={rgb, 255:red, 150; green, 150; blue, 200 }  ,fill opacity=1 ] (476.25,120.47) .. controls (476.25,118.73) and (477.66,117.32) .. (479.4,117.32) .. controls (481.14,117.32) and (482.55,118.73) .. (482.55,120.47) .. controls (482.55,122.21) and (481.14,123.62) .. (479.4,123.62) .. controls (477.66,123.62) and (476.25,122.21) .. (476.25,120.47) -- cycle ;
\draw  [color={rgb, 255:red, 50; green, 50; blue, 100 }  ,draw opacity=1 ][fill={rgb, 255:red, 150; green, 150; blue, 200 }  ,fill opacity=1 ] (476.25,100.47) .. controls (476.25,98.73) and (477.66,97.32) .. (479.4,97.32) .. controls (481.14,97.32) and (482.55,98.73) .. (482.55,100.47) .. controls (482.55,102.21) and (481.14,103.62) .. (479.4,103.62) .. controls (477.66,103.62) and (476.25,102.21) .. (476.25,100.47) -- cycle ;
\draw  [color={rgb, 255:red, 50; green, 50; blue, 100 }  ,draw opacity=1 ][fill={rgb, 255:red, 150; green, 150; blue, 200 }  ,fill opacity=1 ] (476.25,80.47) .. controls (476.25,78.73) and (477.66,77.32) .. (479.4,77.32) .. controls (481.14,77.32) and (482.55,78.73) .. (482.55,80.47) .. controls (482.55,82.21) and (481.14,83.62) .. (479.4,83.62) .. controls (477.66,83.62) and (476.25,82.21) .. (476.25,80.47) -- cycle ;

\draw (140,84.2) node [anchor=north west][inner sep=0.75pt]  [font=\footnotesize] [align=left] {$\displaystyle p_{1} =( 0,1)$};
\draw (162,146) node [anchor=north west][inner sep=0.75pt]   [align=left] {{\footnotesize $\displaystyle p_{2} \ =\ ( 2,\ -1)$}};
\draw (142.69,44) node [anchor=north west][inner sep=0.75pt]  [color={rgb, 255:red, 128; green, 128; blue, 128 }  ,opacity=1 ]  {$N$};
\draw (401.4,43.87) node [anchor=north west][inner sep=0.75pt]  [color={rgb, 255:red, 128; green, 128; blue, 128 }  ,opacity=1 ]  {$M$};
\draw (561.4,147.02) node [anchor=north west][inner sep=0.75pt]  [color={rgb, 255:red, 0; green, 0; blue, 255 }  ,opacity=1 ]  {$\mathfrak{R}_{1}$};
\draw (465,21.47) node [anchor=north west][inner sep=0.75pt]  [color={rgb, 255:red, 0; green, 0; blue, 255 }  ,opacity=1 ]  {$\mathfrak{R}_{2}$};
\draw (286.8,40.2) node [anchor=north west][inner sep=0.75pt]  [color={rgb, 255:red, 208; green, 2; blue, 27 }  ,opacity=1 ]  {$\textcolor[rgb]{0,0,0.78}{\sigma }$};
\draw (545.2,39.8) node [anchor=north west][inner sep=0.75pt]  [color={rgb, 255:red, 0; green, 0; blue, 200 }  ,opacity=1 ]  {$\textcolor[rgb]{0,0,0.78}{\omega }$};
\draw (192.8,121.2) node [anchor=north west][inner sep=0.75pt]  [font=\scriptsize,color={rgb, 255:red, 155; green, 155; blue, 155 }  ,opacity=1 ]  {$0$};
\draw (450.4,122) node [anchor=north west][inner sep=0.75pt]  [font=\scriptsize,color={rgb, 255:red, 155; green, 155; blue, 155 }  ,opacity=1 ]  {$0$};
\draw (385.4,112.72) node [anchor=north west][inner sep=0.75pt]  [font=\tiny]  {$a=\chi ^{( 1,0)}$};
\draw (385.4,93.52) node [anchor=north west][inner sep=0.75pt]  [font=\tiny]  {$b=\chi ^{( 1,1)}$};
\draw (385.4,73.92) node [anchor=north west][inner sep=0.75pt]  [font=\tiny]  {$c=\chi ^{( 1,2)}$};
\draw (531,59.9) node [anchor=north west][inner sep=0.75pt]  [font=\scriptsize,color={rgb, 255:red, 150; green, 150; blue, 200 }  ,opacity=1 ]  {$S_{X}$};

\end{tikzpicture}

\end{center}
The sets of Demazure roots are 
\begin{gather*}
\mathfrak{R}_1 = \{(l, -1) \mid l \in \Zgezero\},\\
\mathfrak{R}_2 = \{(l_1, l_2) \mid 2l_1 - l_2 = -1, \, l_2 \ge 0\} = \{(l, 1 + 2l) \mid l \in \Zgezero\}.
\end{gather*}
\end{example}

\begin{example}
\label{dim3_1example}
Let $X$ be an affine toric variety given by the cone~$\sigma$ with $p_1 = (1, 0, 0)$, $p_2 = (0, 1, 0)$, $p_3 = (1, 0, 1)$ and $p_4 = (0, 1, 1)$. Then the semigroup $S_X = M \cap \,\omega$ is generated by vectors $(1,0,0)$, $(0,1,0)$, $(0,0,1)$ and $(1,1,-1)$. Denoting $a=\chi^{(1,0,0)}$, $b=\chi^{(0,1,0)}$, $c=\chi^{(0,0,1)}$ and $d=\chi^{(1,1,-1)}$, we obtain $\KK[X] = \KK[a, b, c, d]$ with relation $ab = cd$, i.e. the variety $X$ is isomorphic to the quadratic cone $\{ab =  cd\} \subset \AA^4$. 
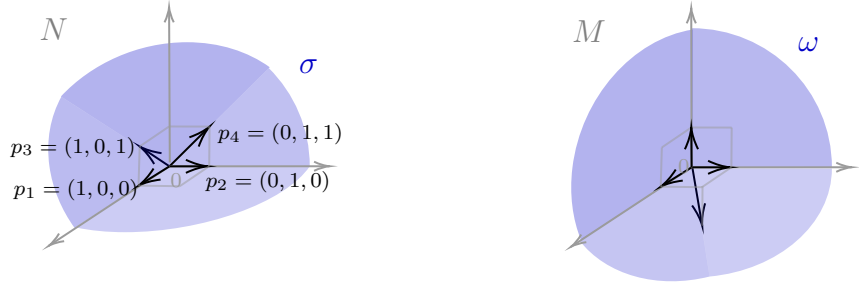
\begin{center}

\tikzset{every picture/.style={line width=0.75pt}} 

\begin{tikzpicture}[x=0.75pt,y=0.75pt,yscale=-1,xscale=1]

\draw    (200,120) -- (186.8,111.12) ;
\draw [shift={(185.14,110)}, rotate = 393.94] [color={rgb, 255:red, 0; green, 0; blue, 0 }  ][line width=0.75]    (10.93,-3.29) .. controls (6.95,-1.4) and (3.31,-0.3) .. (0,0) .. controls (3.31,0.3) and (6.95,1.4) .. (10.93,3.29)   ;
\draw  [draw opacity=0][fill={rgb, 255:red, 0; green, 0; blue, 200 }  ,fill opacity=0.25 ] (249.6,70.61) .. controls (262.2,83.26) and (269.99,100.7) .. (270,119.96) -- (200,120) -- cycle ;
\draw  [draw opacity=0][fill={rgb, 255:red, 0; green, 0; blue, 200 }  ,fill opacity=0.3 ] (145.84,84.85) .. controls (161.03,68.38) and (184.21,57.56) .. (208.76,57.56) .. controls (224.93,57.56) and (239.18,62.25) .. (250.04,70.15) -- (200,120) -- cycle ;
\draw  [draw opacity=0][fill={rgb, 255:red, 0; green, 0; blue, 200 }  ,fill opacity=0.2 ] (270,119.96) .. controls (257.71,138.29) and (216.49,153.14) .. (177.88,153.14) .. controls (168.77,153.14) and (160.62,152.31) .. (153.64,150.81) -- (200.04,119.9) -- cycle ;
\draw  [draw opacity=0][fill={rgb, 255:red, 0; green, 0; blue, 200 }  ,fill opacity=0.27 ] (152.68,151) .. controls (146.16,141.91) and (141.68,131.32) .. (140.11,120) .. controls (138.31,106.98) and (140.63,94.92) .. (146.05,85) -- (200,120) -- cycle ;
\draw    (200,120) -- (218.59,101.41) ;
\draw [shift={(220,100)}, rotate = 495] [color={rgb, 255:red, 0; green, 0; blue, 0 }  ][line width=0.75]    (10.93,-3.29) .. controls (6.95,-1.4) and (3.31,-0.3) .. (0,0) .. controls (3.31,0.3) and (6.95,1.4) .. (10.93,3.29)   ;
\draw [color={rgb, 255:red, 155; green, 155; blue, 155 }  ,draw opacity=1 ][line width=0.75]    (200,120) -- (277.81,119.96) ;
\draw [shift={(279.81,119.96)}, rotate = 539.97] [color={rgb, 255:red, 155; green, 155; blue, 155 }  ,draw opacity=1 ][line width=0.75]    (7.65,-2.3) .. controls (4.86,-0.97) and (2.31,-0.21) .. (0,0) .. controls (2.31,0.21) and (4.86,0.98) .. (7.65,2.3)   ;
\draw    (200,120) -- (217.14,120) ;
\draw [shift={(219.14,120)}, rotate = 180] [color={rgb, 255:red, 0; green, 0; blue, 0 }  ][line width=0.75]    (10.93,-3.29) .. controls (6.95,-1.4) and (3.31,-0.3) .. (0,0) .. controls (3.31,0.3) and (6.95,1.4) .. (10.93,3.29)   ;
\draw [color={rgb, 255:red, 155; green, 155; blue, 155 }  ,draw opacity=1 ][line width=0.75]    (200,120) -- (141.48,158.85) ;
\draw [shift={(139.81,159.96)}, rotate = 326.41999999999996] [color={rgb, 255:red, 155; green, 155; blue, 155 }  ,draw opacity=1 ][line width=0.75]    (7.65,-2.3) .. controls (4.86,-0.97) and (2.31,-0.21) .. (0,0) .. controls (2.31,0.21) and (4.86,0.98) .. (7.65,2.3)   ;
\draw    (200,120) -- (186.26,128.65) ;
\draw [shift={(184.57,129.71)}, rotate = 327.8] [color={rgb, 255:red, 0; green, 0; blue, 0 }  ][line width=0.75]    (10.93,-3.29) .. controls (6.95,-1.4) and (3.31,-0.3) .. (0,0) .. controls (3.31,0.3) and (6.95,1.4) .. (10.93,3.29)   ;
\draw [color={rgb, 255:red, 155; green, 155; blue, 155 }  ,draw opacity=1 ][line width=0.75]    (200.17,118.02) -- (199.82,41.9) ;
\draw [shift={(199.81,39.9)}, rotate = 449.74] [color={rgb, 255:red, 155; green, 155; blue, 155 }  ,draw opacity=1 ][line width=0.75]    (7.65,-2.3) .. controls (4.86,-0.97) and (2.31,-0.21) .. (0,0) .. controls (2.31,0.21) and (4.86,0.98) .. (7.65,2.3)   ;
\draw [color={rgb, 255:red, 155; green, 155; blue, 155 }  ,draw opacity=0.5 ]   (185.14,110) -- (199.83,99.95) ;
\draw [color={rgb, 255:red, 155; green, 155; blue, 155 }  ,draw opacity=0.5 ]   (205.14,130) -- (219.83,119.95) ;
\draw [color={rgb, 255:red, 155; green, 155; blue, 155 }  ,draw opacity=0.5 ]   (199.83,99.95) -- (220,100) ;
\draw [color={rgb, 255:red, 155; green, 155; blue, 155 }  ,draw opacity=0.5 ]   (205.14,130) -- (184.57,129.71) ;
\draw [color={rgb, 255:red, 155; green, 155; blue, 155 }  ,draw opacity=0.5 ]   (184.57,129.71) -- (185.14,110) ;
\draw [color={rgb, 255:red, 155; green, 155; blue, 155 }  ,draw opacity=0.5 ]   (219.83,119.95) -- (220,100) ;
\draw    (460.92,120.74) -- (465.48,147.96) ;
\draw [shift={(465.81,149.94)}, rotate = 260.51] [color={rgb, 255:red, 0; green, 0; blue, 0 }  ][line width=0.75]    (10.93,-3.29) .. controls (6.95,-1.4) and (3.31,-0.3) .. (0,0) .. controls (3.31,0.3) and (6.95,1.4) .. (10.93,3.29)   ;
\draw  [draw opacity=0][fill={rgb, 255:red, 0; green, 0; blue, 200 }  ,fill opacity=0.28 ] (460.38,50.5) .. controls (460.43,50.5) and (460.48,50.5) .. (460.52,50.5) .. controls (499.17,50.5) and (530.5,81.82) .. (530.52,120.47) -- (460.52,120.5) -- cycle ;
\draw  [draw opacity=0][fill={rgb, 255:red, 0; green, 0; blue, 200 }  ,fill opacity=0.23 ] (469.47,175.44) .. controls (443.52,181.99) and (417.88,174.75) .. (404.53,157.24) -- (460.91,120.34) -- cycle ;
\draw  [draw opacity=0][fill={rgb, 255:red, 0; green, 0; blue, 200 }  ,fill opacity=0.2 ] (530.52,121.06) .. controls (530.53,149.09) and (503.89,172.19) .. (469.47,175.44) -- (461.05,121) -- cycle ;
\draw  [draw opacity=0][fill={rgb, 255:red, 0; green, 0; blue, 200 }  ,fill opacity=0.3 ] (405.36,157.21) .. controls (401.12,146.58) and (399.44,134.1) .. (401.04,120.74) .. controls (405.3,85.25) and (431.28,55.91) .. (461.01,51.19) -- (460.92,120.74) -- cycle ;
\draw [color={rgb, 255:red, 155; green, 155; blue, 155 }  ,draw opacity=1 ][line width=0.75]    (460.52,120.5) -- (538.34,120.46) ;
\draw [shift={(540.34,120.46)}, rotate = 539.97] [color={rgb, 255:red, 155; green, 155; blue, 155 }  ,draw opacity=1 ][line width=0.75]    (7.65,-2.3) .. controls (4.86,-0.97) and (2.31,-0.21) .. (0,0) .. controls (2.31,0.21) and (4.86,0.98) .. (7.65,2.3)   ;
\draw    (460.52,120.5) -- (477.67,120.5) ;
\draw [shift={(479.67,120.5)}, rotate = 180] [color={rgb, 255:red, 0; green, 0; blue, 0 }  ][line width=0.75]    (10.93,-3.29) .. controls (6.95,-1.4) and (3.31,-0.3) .. (0,0) .. controls (3.31,0.3) and (6.95,1.4) .. (10.93,3.29)   ;
\draw [color={rgb, 255:red, 155; green, 155; blue, 155 }  ,draw opacity=1 ][line width=0.75]    (460.52,120.5) -- (402,159.36) ;
\draw [shift={(400.34,160.46)}, rotate = 326.41999999999996] [color={rgb, 255:red, 155; green, 155; blue, 155 }  ,draw opacity=1 ][line width=0.75]    (7.65,-2.3) .. controls (4.86,-0.97) and (2.31,-0.21) .. (0,0) .. controls (2.31,0.21) and (4.86,0.98) .. (7.65,2.3)   ;
\draw    (460.52,120.5) -- (446.79,129.15) ;
\draw [shift={(445.1,130.22)}, rotate = 327.8] [color={rgb, 255:red, 0; green, 0; blue, 0 }  ][line width=0.75]    (10.93,-3.29) .. controls (6.95,-1.4) and (3.31,-0.3) .. (0,0) .. controls (3.31,0.3) and (6.95,1.4) .. (10.93,3.29)   ;
\draw [color={rgb, 255:red, 155; green, 155; blue, 155 }  ,draw opacity=0.5 ]   (445.67,110.5) -- (460.36,100.46) ;
\draw [color={rgb, 255:red, 155; green, 155; blue, 155 }  ,draw opacity=0.5 ][line width=0.75]    (465.67,130.5) -- (480.36,120.46) ;
\draw [color={rgb, 255:red, 155; green, 155; blue, 155 }  ,draw opacity=0.5 ]   (460.36,100.46) -- (480.52,100.5) ;
\draw [color={rgb, 255:red, 155; green, 155; blue, 155 }  ,draw opacity=0.5 ]   (465.67,130.5) -- (445.1,130.22) ;
\draw [color={rgb, 255:red, 155; green, 155; blue, 155 }  ,draw opacity=0.5 ]   (445.1,130.22) -- (445.67,110.5) ;
\draw [color={rgb, 255:red, 155; green, 155; blue, 155 }  ,draw opacity=0.5 ]   (480.36,120.46) -- (480.52,100.5) ;
\draw [color={rgb, 255:red, 155; green, 155; blue, 155 }  ,draw opacity=1 ][line width=0.75]    (460.52,120.5) -- (460.34,42.4) ;
\draw [shift={(460.34,40.4)}, rotate = 449.87] [color={rgb, 255:red, 155; green, 155; blue, 155 }  ,draw opacity=1 ][line width=0.75]    (7.65,-2.3) .. controls (4.86,-0.97) and (2.31,-0.21) .. (0,0) .. controls (2.31,0.21) and (4.86,0.98) .. (7.65,2.3)   ;
\draw    (460.52,120.5) -- (460.38,102.46) ;
\draw [shift={(460.36,100.46)}, rotate = 449.53] [color={rgb, 255:red, 0; green, 0; blue, 0 }  ][line width=0.75]    (10.93,-3.29) .. controls (6.95,-1.4) and (3.31,-0.3) .. (0,0) .. controls (3.31,0.3) and (6.95,1.4) .. (10.93,3.29)   ;
\draw [color={rgb, 255:red, 155; green, 155; blue, 155 }  ,draw opacity=0.5 ]   (465.81,149.94) -- (465.97,129.98) ;

\draw (120.83,124.89) node [anchor=north west][inner sep=0.75pt]  [font=\tiny] [align=left] {$\displaystyle p_{1} =( 1,0,0)$};
\draw (216.6,121) node [anchor=north west][inner sep=0.75pt]  [font=\tiny] [align=left] {$\displaystyle p_{2} =( 0,1,0)$};
\draw (133.69,44) node [anchor=north west][inner sep=0.75pt]  [color={rgb, 255:red, 128; green, 128; blue, 128 }  ,opacity=1 ]  {$N$};
\draw (399.4,45.07) node [anchor=north west][inner sep=0.75pt]  [color={rgb, 255:red, 128; green, 128; blue, 128 }  ,opacity=1 ]  {$M$};
\draw (263.24,64.87) node [anchor=north west][inner sep=0.75pt]  [color={rgb, 255:red, 208; green, 2; blue, 27 }  ,opacity=1 ]  {$\textcolor[rgb]{0,0,0.78}{\sigma }$};
\draw (512.4,53.8) node [anchor=north west][inner sep=0.75pt]  [color={rgb, 255:red, 0; green, 0; blue, 200 }  ,opacity=1 ]  {$\textcolor[rgb]{0,0,0.78}{\omega }$};
\draw (199.04,121.85) node [anchor=north west][inner sep=0.75pt]  [font=\tiny,color={rgb, 255:red, 155; green, 155; blue, 155 }  ,opacity=1 ]  {$0$};
\draw (119.37,103.55) node [anchor=north west][inner sep=0.75pt]  [font=\tiny] [align=left] {$\displaystyle p_{3} =( 1,0,1)$};
\draw (222.82,96.77) node [anchor=north west][inner sep=0.75pt]  [font=\tiny] [align=left] {$\displaystyle p_{4} =( 0,1,1)$};
\draw (452.34,114.8) node [anchor=north west][inner sep=0.75pt]  [font=\tiny,color={rgb, 255:red, 155; green, 155; blue, 155 }  ,opacity=1 ]  {$0$};

\end{tikzpicture}

\end{center}
The sets of Demazure roots equal
\begin{gather*}
\mathfrak{R}_1 = \{(-1, l_2, l_3) \mid l_2 \in \Zgezero, l_3 \in \ZZ_{>0}\}, \quad \mathfrak{R}_2 = \{(l_1, -1, l_3) \mid l_1 \in \Zgezero, l_3 \in \ZZ_{>0}\},\\
\mathfrak{R}_3 = \{(l_1, l_2, l_1 + 1) \mid l_1, l_2 \in \Zgezero\}, \quad \mathfrak{R}_4 = \{(l_1, l_2, l_2 + 1) \mid l_1, l_2 \in \Zgezero\}.
\end{gather*}
\end{example}

\section{Affine monoids}
\label{monoid_sec}

\begin{definition}
An \emph{algebraic monoid} is an irreducible algebraic variety~$X$ with an associative multiplication $\mu\colon X \times X \to X$, $(x, y) \mapsto x*y$, which is a morphism of algebraic varieties and admits an element $e \in X$ called the \emph{unit} such that $e*x = x*e = x$ for any $x \in X$.  
\end{definition}

Denote by $G(X)$ the set of invertible elements of a monoid~$X$. It is a connected algebraic group, open in~$X$, see~\cite[Theorem~1]{Ri1}.

\begin{definition}
By a \emph{group embedding} of an algebraic group $G$ we mean an irreducible algebraic variety~$X$ with an open embedding $G \hookrightarrow X$ such that both actions by left and right multiplications of~$G$ on itself can be extended to actions of~$G$ on~$X$. 
\end{definition}

A monoid $X$ is called \emph{affine} if the underlying variety is affine. It is well known that any quasi-affine group is affine, see for example~\cite[Theorem 7.5.3]{FR}. It follows that for affine~$X$, the group $G(X)$ is affine as well. The converse statement is proved in~\cite{Ri2}. 

If $X$ is a monoid, then the embedding $G(X) \hookrightarrow X$ is a group embedding. It is known that for affine monoids the converse is true, i.e. the following proposition holds, see~\cite[Theorem~1]{Vi} for characteristic zero and~\cite[Proposition~1]{Ri1} for the general case. 

\begin{proposition}
Let $G$ be an algebraic group, and $X$ be an affine group embedding of~$G$. Then the product on~$G$ extends to a product $\mu\colon X \times X \to X$ in such a way that $G(X) = G$. 
\end{proposition}

A \emph{morphism} of algebraic monoids $X_1$ and $X_2$ is a morphism $\tau\colon X_1 \to X_2$ of algebraic varieties such that $\tau(x * y) = \tau(x) * \tau(y)$ for any $x,y \in X_1$ and $\tau(e_1) = e_2$ for units $e_1 \in X_1$, $e_2 \in X_2$. 
A \emph{morphism} of group embeddings $X_1$ and $X_2$ of a group $G$ is a morphism $\tau\colon X_1 \to X_2$ of algebraic varieties commuting with $G$-actions on $X_1$ and~$X_2$. 
It is easy to see that these notions are equivalent. 
A morphism~$\tau$ is an \emph{isomorphism} if it is invertible and $\tau^{-1}$ is a morphism as well. 

\medskip

An algebraic monoid $X$ is said to be \emph{commutative} if the multiplication $\mu\colon X \times X \to X$ is commutative, or, equivalently, the group $G(X)$ is commutative. 
It is known \cite[Theorem 15.5]{Hum} that any connected commutative affine group~$G$ over an algebraically closed field~$\KK$ of characteristic zero splits into the product of a torus and a commutative unipotent group, that is $G \cong \GG_m^r \times \GG_a^s$ for some $r, s \in \Zgezero$, where $\GG_m = (\KK^\times, \times)$, $\GG_a = (\KK, +)$ are the multiplicative and additive groups of the ground field. By the \emph{rank} and the \emph{corank} of a commutative monoid $X$ we mean the numbers $r$ and~$s$ for the group of invertible elements $G(X)$, respectively. 

In particular, in the case of corank~$0$ we have group embeddings of the torus $G = \GG_m^n$, i.e. $X$ is an affine toric variety. The monoid structure on~$X$ is unique up to isomorphism since the toric structure on~$X$ is unique, see~\cite{De1982,Gu1998,Be2003}. We call it the \emph{toric monoid structure}. Another remarkable case is the \emph{vector monoid structure} of rank~$0$ on an affine space~$\AA^n$, where the multiplication is the vector addition. It does not depend on the choice of zero point up to isomorphism. 

\smallskip

In the commutative case, there is a natural bijection between group embeddings of~$G$ and effective $G$-actions with an open orbit and a fixed base point in the open orbit. Indeed, any group embedding~$\iota\colon G \hookrightarrow X$ gives a $G$-action on~$X$ with a base point $\iota(e)$ in the open orbit~$\iota(G)$, and for any $G$-action on~$X$ with a base point $x_0$ the orbit map $G \to Gx_0 \subseteq X$, $g \mapsto g \cdot x_0$, is a group embedding of~$G$. 

\medskip

For calculations in Sections~\ref{comult_sec} and~\ref{Coxlang_sec}, we need an explicit description of the correspondence between monoid structures and actions with an open orbit for a commutative group~$G$. It is given in Construction~\ref{monoid_constr} below. 

\begin{construction}
\label{monoid_constr}
Let an affine commutative group~$G$ act effectively on an affine variety~$X$ over an algebraically closed field~$\KK$ of characteristic zero. Let the action have an open orbit and let $x_0$ be chosen as the base point in the open orbit. We are going to construct the corresponding commutative monoid structure on $X$ with the unit~$x_0$. 

Consider the orbit map $\iota\colon G \to Gx_0 \subseteq X$, $g \mapsto g \cdot x_0$. Identifying $G$ with the image $\iota(G) \subseteq X$, we see that the morphism of action $G \times X \to X$ extends the group multiplication $G \times G \to G$. Since the multiplication in $G$ is commutative, the symmetric map $X \times G \to X$ extends the same multiplication $G \times G \to G$. 

Let the homomorphism of algebras $\Phi\colon\KK[G] \to \KK[G] \otimes \KK[G]$ be dual to the multiplication morphism $G \times G \to G$. 
The dual homomorphism to the morphism $\iota\colon G \hookrightarrow X$ gives the inclusion $\KK[X] \hookrightarrow \KK[G]$, so we identify $\KK[X]$ with the subspace of~$\KK[G]$. By the above, the homomorphism~$\Phi$ restricts to the dual homomorphisms $\KK[X] \to \KK[G] \otimes \KK[X]$ and $\KK[X] \to \KK[X] \otimes \KK[G]$. Then the image~$\Phi(\KK[X])$ belongs to
\[
\KK[X] \otimes \KK[X] = (\KK[G] \otimes \KK[X]) \cap (\KK[X] \otimes \KK[G]).
\]
So we obtain the homomorphism $\Phi\colon\KK[X] \to \KK[X] \otimes \KK[X]$ whose dual defines the morphism $\mu\colon X \times X \to X$. The obtained multiplication~$\mu$ is associative, commutative, and admits the unit $x_0$ since the same holds for the multiplication on~$G$. 
\end{construction}

\begin{remark}
The monoid structures on~$X$ corresponding to the same action $G \times X \to X$ and different points $x_0$ in the open orbit are isomorphic since the action of $G$ on the open orbit is transitive. 
\end{remark}

Let us mention that for any irreducible algebraic monoid there exists a unique algebraic monoid structure on its normalization such that the normalization map is a morphism of algebraic monoids, see~\cite[Proposition~3.15]{Re}. This motivates us to focus on monoid structures on normal varieties. 

\section{The comultiplication language}
\label{comult_sec}

Let $X$ be an affine algebraic variety. We are going to describe multiplications $X \times X \to X$ in terms of their comorphisms $\Phi\colon \KK[X] \to \KK[X] \otimes \KK[X]$, which we call \emph{comultiplications}.  

\medskip

Let $G$ be an affine algebraic group. We need the following lemma. 
\begin{lemma}
\label{techcomult_lem}
Let $\phi\colon G \times \KK[X] \to \KK[X]$ be the action on the algebra $\KK[X]$ corresponding to the action of~$G$ on~$X$. Then the map $\Phi\colon \KK[X] \to \KK[G] \otimes \KK[X]$ dual to the action $G \times X \to X$ maps $f \in \KK[X]$ to $\phi(g^{-1}, f)$ considered as a function in the argument~$g \in G$. 
\end{lemma}

\begin{proof}
If $\phi\colon G \times \KK[X] \to \KK[X]$, $(g, f) \mapsto \phi(g, f)$, is the action on~$\KK[X]$, then the action $G \times X \to X$ maps a pair $(g, x)$ to a point $\phi(g^{-1}, f)(x)$ considered as an evaluation homomorphism $\KK[X] \to \KK$ in the argument~$f$. Then its dual is as claimed. 
\end{proof}

\begin{theorem}
\label{comult_theor}
Let $X$ be a normal affine algebraic variety of dimension~$n$ with commutative monoid structure of rank $r = 0,\, n-1$ or~$n$. Then $X$ is toric, so $\KK[X] = \!\bigoplus\limits_{u \in S_X} \!\KK\chi^u$, and we have the following classification of monoid structures up to isomorphism. 

1) If $r = n$, then $X$ is endowed with the toric monoid structure and the comultiplication ${\KK[X] \to \KK[X] \otimes \KK[X]}$ is given by
\[\chi^u \mapsto \chi^u \otimes \chi^u \;\text{ for any } u \in S_X.\]

2) If $r = n-1$, then the comultiplication $\KK[X] \to \KK[X] \otimes \KK[X]$ is defined by
\[\chi^u \mapsto \chi^u \otimes \chi^u (1 \otimes \chi^e + \chi^e \otimes 1)^{\langle p_i, u\rangle} \;\text{ for any } u \in S_X,\] 
where $p_i$ is the primitive vector on a ray of the fan of $X$ and $e \in \mathfrak{R}_i$ is a Demazure root corresponding to this ray.

3) If $r = 0$, then $X = \AA^n$ and $X$ has the vector monoid structure with the comultiplication
\[\chi^u \mapsto \chi^u \otimes 1 + 1 \otimes \chi^u \;\text{ for any } u \in \Zgezero^n.\]
\end{theorem}

\begin{remark}
For a Demazure root $e \in \mathfrak{R}_i$, we have $\langle p_i, e\rangle = -1$, whence $e \notin S_X$ and a function $\chi^e \in \KK(X)$ does not belong to $\KK[X]$. However, $\chi^{u + je} \in \KK[X]$ for any $0 \le j \le \langle p_i, u\rangle$, so the map in 2) is well defined. 
\end{remark}

\begin{proof}
1) We preserve the notation of previous sections. 
In the case $r=n$, the monoid structure is given by a group embedding of the torus $G = \GG_m^n$, i.e. the variety $X$ is toric. Since the toric structure on~$X$ is unique, see~\cite{De1982,Gu1998,Be2003}, the monoid structure on~$X$ is unique as well. According to Construction~\ref{monoid_constr}, the toric multiplication extends the torus action $T \times X \to X$ with an open orbit. The corresponding action $\phi\colon T \times \KK[X] \to \KK[X]$ is given by $(t, \chi^u) \mapsto t^u \chi^u$ for any $u \in M$. By Lemma~\ref{techcomult_lem}, the map $\Phi\colon \KK[X] \to \KK[T] \otimes \KK[X]$ is given by $\chi^u \mapsto t^{-u} \chi^u$. According to Construction~\ref{monoid_constr}, we identify $T$ with the open subset of $X$ by acting at a point $x_0$ in the open orbit. Take $x_0 \in X$ such that $\chi^u(x_0) = 1$ for any $u \in M$. We obtain that $\KK[X] \hookrightarrow \KK[T]$ is given by $\chi^u \mapsto t^{-u} \chi^u(x_0) = t^{-u}$, so the comorphism $\Phi$ above restricts to the comultiplication $\Phi\colon\KK[X] \to \KK[X] \otimes \KK[X]$, $\chi^u \mapsto \chi^u \otimes \chi^u$. 

\smallskip

2) In case of corank~$1$, we have a group embedding of $G = \GG_m^{n-1} \times \GG_a$. First let us construct some structures on an affine toric variety~$X$ with acting torus~$T$. 

Let $e \in \mathfrak{R}_i$ be a Demazure root. Since the one-parameter group $\GG_a$ corresponding to the LND $\pa_e$ is normalized by the torus $T$, the group $T \rightthreetimes \GG_a$ acts on $X$. Then the group $G = \Ker \chi^e \times \GG_a$ has corank~$1$, $G$ acts on $X$ with an open orbit, and $X$ is a group embedding of~$G$, see~\cite[Proposition 6]{AK}. 

By~\cite[Theorem 2]{AK}, all group embeddings of the group~$G = \GG_m^{n-1} \times \GG_a$ can be realized this way, i.e. $X$ is toric and $G = \Ker \chi^e \times \GG_a$. Since the LND $\pa_e$ maps $\chi^u \mapsto \langle p_i, u\rangle \chi^{u+e}$, we get the formula for the action $\phi\colon G \times \KK[X] \to \KK[X]$: an element $(t, \alpha) \in G$ acts on $\chi^u \in \KK[X]$ by $t\exp \alpha\pa_e$, and the result equals
\begin{equation}
\label{rk1act_eq}
t^u \left(\chi^u + \alpha \langle p_i, u\rangle\chi^{u+e} + 
\frac{\alpha^2}{2} \langle p_i, u\rangle \langle p_i, u + e\rangle \chi^{u+2e} + \ldots \right)
= 
t^u \sum_{j = 0}^{\langle p_i, u\rangle} \binom{\langle p_i, u\rangle}{j} \alpha^j \chi^{u + je},
\end{equation}
since $\langle p_i, e\rangle = -1$. By Lemma~\ref{techcomult_lem}, the map $\Phi\colon \KK[X] \to \KK[G] \otimes \KK[X]$ is given by
\[
\Phi(\chi^u) = t^{-u} \sum_{j = 0}^{\langle p_i, u\rangle} \binom{\langle p_i, u\rangle}{j} (-\alpha)^j \chi^{u + je}.
\]
According to Construction~\ref{monoid_constr}, consider an embedding $G \hookrightarrow X$ given by a point $x_0$ in the open orbit. Take $x_0$ with
\[\chi^u(x_0) = 
\begin{cases}
1 \; \text{ if } \; \langle p_i, u\rangle = 0 \\
0 \; \text{ if } \; \langle p_i, u\rangle > 0 
\end{cases}.\] 
The evaluation of~\eqref{rk1act_eq} at $x_0$ yields that the dual embedding $\KK[X] \hookrightarrow \KK[G]$ is given by $(t, \alpha) \mapsto t^{-u} (-\alpha)^{\langle p_i, u\rangle}$, $(t, \alpha) \in G$. In particular, for $0 \le j \le \langle p_i, u\rangle$ the function $\chi^{u + ke}$ is identified with $t^{-u}(-\alpha)^{\langle p_i, u + je\rangle} = t^{-u}(-\alpha)^{\langle p_i, u\rangle - j}$. Thus, the formula for the comultiplication $\Phi\colon \KK[X] \to \KK[X] \otimes \KK[X]$ is
\[
\Phi(\chi^u) = \sum_{j = 0}^{\langle p_i, u\rangle} \binom{\langle p_i, u\rangle}{j} \chi^{u + (\langle p_i, u\rangle - j)e} \otimes \chi^{u + je} = \chi^u \otimes \chi^u (1 \otimes \chi^e + \chi^e \otimes 1)^{\langle p_i, u\rangle},
\]
as required.

\smallskip

3) Any monoid structure of rank~$0$ is given by an effective action of the group~$\GG_a^n$ with an open orbit. Since any orbit of a unipotent group on an affine variety~$X$ is closed~\cite[Section 1.3]{PV}, an open orbit of $\GG_a^n$ has to coincide with the monoid $X$. By Construction~\ref{monoid_constr}, the multiplication on~$X$ equals the operation in $\GG_a^n$. Thus $X \cong \AA^n$ is endowed with the vector monoid structure, the multiplication is given by the formula $(x, y) \mapsto x + y$, and the dual is $\chi^u \mapsto \chi^u \otimes 1 + 1 \otimes \chi^u$.
\end{proof}

Let us formulate when the above monoid structures are isomorphic. Since the isomorphism of monoids respects the group of invertible elements, monoids of different ranks are not isomorphic. Proposition~10 of~\cite{AK} claims that group embeddings of $\GG_m^{n-1} \times \GG_a$ given by two Demazure roots $e_1, e_2$ of a toric variety are equivalent if and only if there is an automorphism of the lattice $N$ which preserves the fan and such that the induced automorphism of the dual lattice $M$ sends $e_2$ to $e_1$. Thus we have the following corollary of~\cite[Proposition~10]{AK}. 

\begin{proposition}
\label{isomorph_dimn_prop}
In notation of Theorem~\ref{comult_theor}, monoids or rank~$n-1$ corresponding to Demazure roots $e_1, e_2 \in \mathfrak{R}$ are isomorphic if and only if there is an automorphism of the lattice~$N$ which preserves the fan of~$X$ and such that the induced automorphism of the dual lattice~$M$ sends $e_2$ to~$e_1$.
\end{proposition}

\begin{example}
\label{dim3_multexample}
Let us list all monoid structures of rank~2 corresponding to $p_1 = (1,0,0)$ for the quadratic cone $X = \{ab = cd\} \subseteq \AA^4$, see Example~\ref{dim3_1example}. According to Theorem~\ref{comult_theor}, the comultiplication is given by $\chi^u \mapsto \chi^u \otimes \chi^u (1 \otimes \chi^e + \chi^e \otimes 1)^{u_1}$. Recall that $a=\chi^{(1,0,0)}$, $b=\chi^{(0,1,0)}$, $c=\chi^{(0,0,1)}$ and $d=\chi^{(1,1,-1)}$, thus for a Demazure root $e = (-1, l_2, l_3) \in \mathfrak{R}_1$ we have $\chi^e(x) = a(x)^{-1} b(x)^{l_2} c(x)^{l_3}$. Then the coordinates $a, b, c, d$ of the product $x*y$ are given by
\begin{gather*}
a(x*y) = a(x) b(y)^{l_2} c(y)^{l_3} + b(x)^{l_2} c(x)^{l_3} a(y),\\
b(x*y) = b(x) b(y), \quad
c(x*y) = c(x) c(y),\\
d(x*y) = d(x) b(y)^{l_2+1} c(y)^{l_3-1} + b(x)^{l_2+1} c(x)^{l_3-1} d(y).
\end{gather*}
Since any ray of the cone can be mapped to any other ray by an automorphism of the lattice~$N$ preserving~$\sigma$, the above monoid structures are all the monoid structures of rank~$1$ on~$X$ by Proposition~\ref{isomorph_dimn_prop}. 
\end{example}

\begin{example}
Let $X$ be an affine cone over the singular del Pezzo surface $Y$ of type $D_5$ of degree~$4$, i.e. 
\[X = \{x_0x_1 - x_2^2 = x_0x_4 - x_1x_2 + x_3^2 = 0\} \subseteq \AA^5.\]
According to \cite[Lemma~6]{DL}, $Y$ admits an action of the additive group $\GG_a^2$ with an open orbit, whence $X$ admits an action of the group $\GG_m \times \GG_a^2$ with an open orbit. Then $X$ has a commutative monoid structure of rank~$1$. It is known that $Y$ is a non-toric variety and the Cox ring $\RRR(Y)$ is not a polynomial ring, see~\cite[Theorem~5.4.4.2, no.~7]{ADHL}. Let us notice that $\RRR(Y) = \RRR(X)$. Indeed, \cite[Construction~4.2.1.2 and Example 4.2.1.6]{ADHL} imply that the quotient presentation $\widehat Y \xrightarrow{/\!/ H_Y} Y$ for characteristic space~$\widehat Y$ factors through the quotient presentation $X \xrightarrow{/\!/ \KK^\times} Y$, and the map $\widehat Y \to X$ is a quotient presentation, which is a characteristic space for~$X$ by~\cite[Theorem~1.6.4.3]{ADHL}. Then $\RRR(X) = \RRR(Y)$ is the algebra of regular functions on $\widehat Y$ by~\cite[Construction~1.6.1.3]{ADHL}. Thus $\RRR(X)$ is not a polynomial ring as well, whence $X$ is non-toric. This shows that in dimension~$3$ the variety admitting a commutative monoid structure need not be toric. 
\end{example}

\section{The Cox construction}
\label{Cox_sec}

We preserve the notation of Sections~\ref{toric_sec} and~\ref{monoid_sec} and describe the Cox construction for affine toric varieties. See~\cite{Cox},~\cite[Chapter 5]{CLS} for arbitrary toric varieties and~\cite{ADHL} for general varieties. 

\begin{construction}
\label{Cox1_constr}
Let $X$ be an affine toric variety with the acting torus~$T$. First let us assume that $\sigma(1)$ spans $N_\QQ$. This means that $X$ has no torus factors, or $\KK[X]$ has no non-constant invertible functions. 

Each ray $\rho_i \in \sigma(1)$, $1 \le i \le m$, corresponds to an irreducible $T$-invariant Weil divisor~$D_i$ on~$X$. The free abelian group $\barM$ of $T$-invariant Weil divisors on~$X$ is freely generated by divisors~$D_i$, $1 \le i \le m$, i.e. a divisor $D \in \barM \simeq \ZZ^m$ is uniquely represented as a sum $\sum_{i=1}^m a_iD_i$, $a_i \in \ZZ$. Each rational function $\chi^u \in \KK(X)$, $u \in M$, gives the principal divisor $\Div(\chi^u) = \sum \langle p_i, u\rangle D_i$, see~\cite[Section~3.3]{Fu}. 

Denote by~$\Cl(X)$ the class group of~$X$. By \cite[Section~3.4]{Fu}, there is an exact sequence
\begin{equation}
\label{Coxexact_eq}
0 \to M \to \barM \to \Cl(X) \to 0,
\end{equation}
where $M \to \barM \simeq \ZZ^m$ is given by $u \mapsto \Div(\chi^u) \simeq \bar u := (\langle p_1, u\rangle, \ldots, \langle p_m, u\rangle)$ and $\barM \to \Cl(X)$ maps a divisor~$D$ to the class $[D]$ in the class group. 

The \emph{Cox ring} of a toric variety $X$ is a polynomial ring~$\RRR(X) = \KK[x_1, \ldots, x_m]$ graded by the group~$\Cl(X)$ via $\deg x_i = [D_i]$ and the \emph{total coordinate space} is the affine space $\barX := \Spec \RRR(X) = \AA^m$. 

Let a quasitorus $H_X := \Hom(\Cl(X), \KK^\times)$ and a torus $\barT := \Hom(\barM, \KK^\times) \simeq (\KK^\times)^m$ be the character groups of the abelian group $\Cl(X)$ and the lattice $\barM$, respectively. The application of $\Hom(\,\cdot\,, \KK^\times)$ to~\eqref{Coxexact_eq} gives the dual exact sequence
\begin{equation}
\label{Coxexact2_eq}
1 \leftarrow T \xleftarrow{\,\pi'} \barT \leftarrow H_X \leftarrow 1.
\end{equation}
Since~$M \to \barM$ is given by $u \mapsto \bar u$, the dual map $(\pi')^*\colon \KK[T] \to \KK[\barT]$ is given by $\chi^u \mapsto \chi^{\bar u}$, $u \in M$. 

The diagonal action of $\barT \simeq (\KK^\times)^m$ on the total coordinate space~$\barX = \AA^m$ and the inclusion $H_X \hookrightarrow \barT$ in~\eqref{Coxexact2_eq} define an action of $H_X$ on $\AA^m$. This action corresponds to the $\Cl(X)$-grading on~$\RRR(X)$. 

Let us fix some group embeddings of tori $T$ and $\barT$ in $X$ and $\barX$, respectively. Then the map $(\pi')^*\colon \KK[T] \to \KK[\barT] = \KK[x_1^{\pm1}, \ldots, x_m^{\pm1}]$ restricts to an isomorphism between $\KK[X]$ and the algebra of invariants $\KK[x_1, \ldots, x_m]^{H_X} \subseteq \KK[x_1, \ldots, x_m]$ defined by $\chi^u \mapsto \chi^{\bar u}$, $u \in S_X$. This means that $X$ is the categorical quotient of~$\barX = \AA^m$ by the above action of~$H_X$, so there is the canonical \emph{Cox construction} $\pi\colon \barX \xrightarrow{/\!/ H_X} X$, see~\cite[Theorem~2.1, Lemma 2.2]{Cox}. 
\end{construction}

\begin{example}
\label{cone_2example}
Let $X$ be the affine toric surface given by the cone~$\sigma$ with $p_1 = (0, 1)$, $p_2 = (2, -1)$, see Example~\ref{cone_1example}. The group of invariant Weil divisors is freely generated by the corresponding divisors~$D_1, D_2$. The relations between $[D_1], [D_2] \in \Cl(X)$ are given by principal divisors $\Div(\chi^u) = \langle p_1, u\rangle D_1 + \langle p_2, u\rangle D_2$ for basis vectors $u$, so we obtain 
$2[D_1] = 0$ and $[D_1] - [D_2] = 0$, i.e. $\Cl(X) = \ZZ/2\ZZ$. 
The Cox ring is $\KK[x_1, x_2]$ graded by $\Cl(X)$ via $\deg x_1 = [D_1]$, $\deg x_2 = [D_2]$, and the finite group $H_X \cong \ZZ/2\ZZ$ acts on the total coordinate space $\barX = \AA^2$ via $(x_1, x_2) \mapsto (-x_1, -x_2)$. Then the algebra of invariants equals $\KK[x_1, x_2]^{H_X} = \KK[x_1^2, x_1x_2, x_2^2]$, so the categorical quotient $\AA^2 /\!/ H_X$ is the quadratic cone $X= \{ac = b^2\} \subseteq \AA^3$, where the regular functions~
\[a = \chi^{(1,0)}, \; b = \chi^{(1,1)}, \; c = \chi^{(1,2)} \text{ in } \KK[X]\]
correspond to 
\[\chi^{\overbar{(1,0)}} = \chi^{(0,2)}=x_2^2, \; \chi^{\overbar{(1,1)}} = \chi^{(1,1)}=x_1x_2, \; \chi^{\overbar{(1,2)}} = \chi^{(2,0)}=x_1^2\phantom{x_3} \text{ in } \KK[\barX],\]
respectively. 
\end{example}

\begin{example}
\label{dim3_2example}
Let $X$ be the affine toric hypersurface from Example~\ref{dim3_1example} given by the cone~$\sigma$ with $p_1 = (1, 0, 0)$, $p_2 = (0, 1, 0)$, $p_3 = (1, 0, 1)$, $p_4 = (0, 1, 1)$. The group of invariant Weil divisors is freely generated by the corresponding divisors~$D_1, \ldots, D_4$. The relations between $[D_1], \ldots, [D_4]$ in the class group~$\Cl(X)$ are given by $[\Div(\chi^u)] = 0$ for basis vectors $u$, so we obtain 
\[
\begin{pmatrix}
	1 & 0 & 1 & 0 \\ 0 & 1 & 0 & 1 \\ 0 & 0 & 1 & 1
\end{pmatrix}
\begin{pmatrix}
	[D_1] \\ [D_2] \\ [D_3] \\ [D_4]
\end{pmatrix}
= \begin{pmatrix} 0 \\ 0 \\ 0 \end{pmatrix}
\] 
and $\Cl(X)$ is a free abelian group or rank~$1$ generated by $[D_1] = -[D_2] = -[D_3] = [D_4]$. 
The Cox ring is $\KK[x_1, x_2, x_3, x_4]$ graded by $\Cl(X)$ via $\deg x_i = [D_i]$, whence the one-dimensional torus ${H_X = \KK^{\times}}$ acts on $\barX = \AA^4$ via $t \cdot (x_1, x_2, x_3, x_4) = (t x_1, t^{-1} x_2, t^{-1} x_3, tx_4)$. Then $\KK[x_1, x_2, x_3, x_4]^{H_X} = \KK[x_1x_2, x_1x_3, x_2x_4, x_3x_4]$, so the quotient~$\AA^4 /\!/ H_X$ is the quadratic cone $X= \{ab = cd\}$, where 
\[a = \chi^{(1,0)}, \; b = \chi^{(1,1)}, \; c = \chi^{(1,2)} \text{ in } \KK[X]\]
correspond to 
\[\chi^{\overbar{(1,0)}} = \chi^{(0,2)}=x_2^2, \; \chi^{\overbar{(1,1)}} = \chi^{(1,1)}=x_1x_2, \; \chi^{\overbar{(1,2)}} = \chi^{(2,0)}=x_1^2\phantom{x_3} \text{ in } \KK[\barX],\]
respectively. 
\end{example}

\begin{definition}
Let $\pi\colon \barX \xrightarrow{/\!/H_X} X$ be the Cox construction and let $G, \barG$ be algebraic groups. We say that an action $\mu\colon G \times X \to X$ \emph{is coherent} with an action $\bar\mu\colon\barG \times \barX \to \barX$ if there exists an epimorphism $\pi'\colon \barG \to G$ such that the actions of $\barG$ and $H_X$ on $\barX$ commute and the following diagram is commutative. 
\[
\begin{diagram}
\node{\barG \times \barX} \arrow[2]{e,t}{\bar \mu} \arrow{s,r}{\pi'\times\pi} \node[2]{\barX} \arrow{s,r}{\pi} \\
\node{G \times X} \arrow[2]{e,t}{\mu} \node[2]{X}
\end{diagram}\]
\end{definition}

\smallskip

The following three lemmas allow us to ``lift'' actions on~$X$ to coherent actions on~$\barX$. 

\begin{lemma}
\label{lifttorus_lem}
The action of the torus $T$ on $X$ is coherent with the action of the torus $\barT$ on~$\barX$.
\end{lemma}

\begin{proof}
Since the action of $H_X$ comes from the action of $\barT$ via $H_X \hookrightarrow \barT$ in~\eqref{Coxexact2_eq} and $\barT$ is commutative, the actions of $\barT$ and $H_X$ on $\barX$ commute as well. 

Let $\pi'$ be the epimorphism in~\eqref{Coxexact2_eq}. Notice that $\chi^u(t) = \chi^{\bar u}(\bar t)$ if $u \in M$, $\bar t \in \barT$, $t = \pi'(\bar t)$. This easily follows from the definitions of $T = \Hom(M, \KK^\times)$, $\barT = \Hom(\barM, \KK^\times)$, and the map $\pi'\colon \barT \to T$. 

Commutativity of the first diagram below is equivalent to commutativity of the second one for any $\bar t \in \barT$. The latter one follows from $\chi^u(t) = \chi^{\bar u}(\bar t)$ and the fact that the isomorphism $\pi^*\colon \KK[X] \hookrightarrow \KK[\barX]^{H_X} \subseteq \KK[\barX]$ maps $\chi^u$ to $\chi^{\bar u}$: 
\[
\begin{diagram}
\node{\barT \times \barX} \arrow[2]{e,t}{\bar \mu} \arrow{s,r}{\pi'\times\pi} \node[2]{\barX} \arrow{s,r}{\pi} \\
\node{T \times X} \arrow[2]{e,t}{\mu} \node[2]{X}
\end{diagram} \quad\quad 
\begin{diagram}
\node{\KK[\barX]} \node[2]{\KK[\barX]} \arrow[2]{w,t}{\bar \mu(\bar t, \,\cdot\,)^*} \\
\node{\KK[X]} \arrow{n,r,L}{\pi^*} \node[2]{\KK[X]} \arrow[2]{w,t}{\mu(\pi'(\bar t), \,\cdot\,)^*} \arrow{n,r,L}{\pi^*}
\end{diagram} \quad\quad
\begin{diagram}
\node{\chi^{\bar u}(\bar t)\chi^{\bar u}} \node[2]{\chi^{\bar u}} \arrow[2]{w,t,T}{\bar t} \\
\node{\chi^u(t)\chi^u} \arrow{n,r,T}{\pi^*}  \node[2]{\chi^u} \arrow[2]{w,t,T}{t} \arrow{n,r,T}{\pi^*}
\end{diagram}
\]
\end{proof}

\begin{lemma}
\label{liftGm_lem}
The action of the subtorus $\Ker \chi^u \subset T$ on $X$ corresponding to any $u \in M$ is coherent with the action of the subtorus $\Ker \chi^{\bar u} \subset \barT$ on $\barX$ corresponding to $\bar u \in \barM$. 
\end{lemma}

\begin{proof}
The actions of $H_X$ and $\barT$ on $\barX$ commute by Lemma~\ref{lifttorus_lem}, whence the actions of $\Ker \chi^{\bar u} \subset \barT$ and $H_X$ on $\barX$ commute as well. 

Let $\pi'$ be the epimorphism in~\eqref{Coxexact2_eq}. We have $\chi^u(t) = \chi^{\bar u}(\bar t)$ for $u \in M$, $\bar t \in \barT$, and $t = \pi'(\bar t)$, so $\Ker \chi^{\bar u}$ is the preimage of $\Ker \chi^u$ and the exact sequence~\eqref{Coxexact2_eq} restricts to
\[1 \leftarrow \Ker \chi^u \xleftarrow{\pi'} \Ker \chi^{\bar u} \leftarrow H_X \leftarrow 1.\] 
Since the diagram for $\barT$ and $T$ is commutative by Lemma~\ref{lifttorus_lem}, the required diagram with $\Ker \chi^{\bar u}$ and $\Ker \chi^u$ is commutative as well. 
\end{proof}

\smallskip

Let us notice that if $e \in \mathfrak{R}_i$ is a Demazure root of $X$, then $\bar e = (\langle p_1, e\rangle, \ldots, \langle p_m, e\rangle)$ is a Demazure root of~$\barX = \AA^m$ since $\langle p_i, e\rangle = -1$ and $\langle p_j, e\rangle \ge 0$, $j \ne i$, see Example~\ref{A2_example}. 

\begin{lemma}
\label{liftGa_lem}
The $\GG_a$-action on $X$ corresponding to a Demazure root $e \in M$ is coherent with the $\GG_a$-action on $\barX$ corresponding to the Demazure root $\bar e \in \barM$. 
\end{lemma}

\begin{proof}
The actions of $H_X$ and $\GG_a = \{\exp s\pa_{\bar e} \mid s \in \KK\}$ on $\barX = \AA^m$ commute according to the exact sequence in the proof of the previous lemma: $H_X$ acts as a subgroup of $\Ker \chi^{\bar e}$, which commutes with the $\GG_a$-action corresponding to $\bar e$. 

Let $\pi'$ be the isomorphism whose dual identifies $\exp s \pa_{\bar e}$ with $\exp s \pa_e$. It is sufficient to check commutativity of the following diagram:
\[
\begin{diagram}
\node{\KK[\barX]} \node[2]{\KK[\barX]} \arrow[2]{w,t}{\pa_{\bar e}} \\
\node{\KK[X]} \arrow{n,r,L}{\pi^*}  \node[2]{\KK[X]} \arrow[2]{w,t}{\pa_e} \arrow{n,r,L}{\pi^*}
\end{diagram}
\]

\medskip

\noindent This holds since $\pi^*$ identifies $\pa_e(\chi^u) = \langle p_i, u\rangle \chi^{u + e}$ with $\pa_{\bar e}(\chi^{\bar u}) = \langle e_i, \bar u\rangle \chi^{\bar u + \bar e}$, where $\langle e_i, \bar u\rangle = \langle p_i, u\rangle$ is the $i$-th coordinate of $\bar u \in \barM = \ZZ^m$. 
\end{proof}

\begin{construction}
\label{Cox2_constr}
Now assume that $\sigma(1)$ does not span $N_\QQ$, i.e. $\sigma$ belongs to a subspace of some codimension $\widetilde m$ spanning the lattice $N_0 := \Span(\sigma) \cap N$. 
This means that $X$ splits into the direct product of a torus $(\KK^\times)^{\widetilde m}$ and a toric variety $X_0$ whose cone $\sigma$ spans~$(N_0)_\QQ$. 
Then we have the above Cox construction for $X_0$, i.e. $X_0$ is the categorical quotient of an affine space $\overbar{X_0} = \AA^m$ by the action of the quasitorus $H_X := \Hom(\Cl(X_0), \KK^\times) = \Hom(\Cl(X), \KK^\times)$. 
Consider the trivial extension of this action to the action on $\barX := \AA^m \times (\KK^\times)^{\widetilde m}$. The categorical quotient by the action of the quasitorus $H_{X}$ is the given variety $X = X_0 \times (\KK^\times)^{\widetilde m}$. More precisely, let us denote the Cox coordinates on $\overbar{X_0} = \AA^m$ by $x_i$, $1 \le i \le m$, and the coordinates on~$(\KK^\times)^{\widetilde m}$ by $x_i \in \KK^\times$, $m < i \le m + \widetilde m$. 
Then we have $\KK[\barX] = \KK[x_1, \ldots, x_m, x_{m+1}^{\pm 1}, \ldots, x_{m+\widetilde m}^{\pm 1}]$, and
the ring of invariants equals \[\KK[\barX]^{H_X} = \KK[x_1, \ldots, x_m]^{H_X}[x_{m+1}^{\pm 1}, \ldots, x_{m+\widetilde m}^{\pm 1}] \cong \KK[X_0][x_{m+1}^{\pm 1}, \ldots, x_{m+\widetilde m}^{\pm 1}] = \KK[X],\] 
i.e. the inclusion of algebras dual to the categorical quotient is given by $\chi^u \mapsto \chi^{\bar u}$, where 
\begin{gather*}
u = u_0 + \tilde u \in S_X = S_{X_0} + \ZZ^{\widetilde m} \subseteq M = M_0 + \ZZ^{\widetilde m}, \\ 
\bar u = (\langle p_1, u\rangle, \ldots, \langle p_m, u\rangle, \tilde u_{m+1}, \ldots, \tilde u_{m+\widetilde m}) \in \barM = \overbar{M_0} + \ZZ^{\widetilde m}.
\end{gather*}

Denote by $T_0$ the torus acting on $X_0$ with an open orbit. One can easily see that $T = T_0 \times (\KK^\times)^{\widetilde m}$ and $\barT = \overbar{T_0} \times (\KK^\times)^{\widetilde m}$, and 
the sequences~\eqref{Coxexact_eq} and~\eqref{Coxexact2_eq} are exact as well. Then Lemma~\ref{lifttorus_lem} and Lemma~\ref{liftGm_lem} hold for $\barX = \overbar{X_0} \times (\KK^\times)^{\widetilde m}$. 

Finally, let us notice that the set of Demazure roots $\mathfrak{R}_i \subseteq M$ of $X$ splits into the direct sum of the set of Demazure roots $\mathfrak{R}_{i,0} \subseteq M_0$ of $X_0$ and $\ZZ^{\widetilde m}$. 
According to~Example~\ref{barX_example}, for a Demazure root $e \in M$ of $X$, an element $\bar e \in \barM$ is a Demazure root of $\barX = \AA^m \times (\KK^\times)^{\widetilde m}$. Then Lemma~\ref{liftGa_lem} holds for $\barX = \overbar{X_0} \times (\KK^\times)^{\widetilde m}$ as well. 
\end{construction}

\section{The Cox construction language}
\label{Coxlang_sec}

Let $X$ be a normal affine toric variety and $\pi\colon \barX \to X$ be the Cox construction described in Section~\ref{Cox_sec}. 

\begin{definition}
\label{coher_mon_def}
Let $X$ and $\barX$ admit monoid structures. A multiplication ${\mu\colon X \times X \to X}$ is said to be \emph{coherent} with a multiplication $\bar \mu\colon \barX \times \barX \to \barX$ if the following diagram is commutative. 
\[
\begin{diagram}
\node{\barX \times \barX} \arrow[2]{e,t}{\bar \mu} \arrow{s,r}{\pi\times\pi} \node[2]{\barX} \arrow{s,r}{\pi} \\
\node{X \times X} \arrow[2]{e,t}{\mu} \node[2]{X}
\end{diagram}\]
\end{definition}

\medskip

\begin{lemma}
\label{act_mon_coh_lem}
Assume that $G$ and $\barG$ are commutative affine algebraic groups acting coherently with open orbits on $X$ and $\barX$, respectively. Let $\bar x_0 \in \barX$ be a point in the open orbit and $x_0 = \pi(\bar x_0)$. Consider the multiplications on $X$ and $\barX$ corresponding to these actions and base points, see Construction~\ref{monoid_constr}. Then these multiplications are coherent. 
\end{lemma}

\begin{proof}
Since the actions are coherent and $x_0 = \pi(\bar x_0)$, the diagram
\[
\begin{diagram}
\node{\barG \parbox[t][4pt]{0pt}{}} \arrow[2]{e,t,J}{} \arrow{s,r}{\pi'} \node[2]{\barX\parbox[t][4pt]{0pt}{}} \arrow{s,r}{\pi} \\
\node{G\parbox[t][4pt]{0pt}{}} \arrow[2]{e,t,J}{} \node[2]{X\parbox[t][4pt]{0pt}{}}
\end{diagram}\]
\parbox[b][15pt]{0pt}{}is commutative, i.e. $\pi$ is an extension of $\pi'$ after the identifications of the groups with their open orbits. Then coherency of the actions implies coherency of the multiplications. 
\end{proof}

\begin{theorem}
\label{Cox_theor}
Let $X$ be a normal affine algebraic variety of dimension~$n$ with monoid structure of rank $0$, $n-1$ or~$n$. Then $X$ is toric, and the multiplication on~$X$ is coherent with a multiplication $\bar x \cdot \bar y$ on~$\barX = \AA^m \times (\KK^\times)^{\widetilde m}$. Up to isomorphism of monoid structures on~$X$ and~$\barX$, the product $\bar x \cdot \bar y$ is given by the following formulas in coordinates $x_i$, $1 \le i \le m + \widetilde m$. 

1) If $r = n$, then
\[\bar x \cdot \bar y = (x_1y_1, \ldots, x_{m + \widetilde m}y_{m + \widetilde m}).\]

2) If $r = n-1$, then
\[\bar x \cdot \bar y = 
(x_1y_1, \ldots, x_{i-1}y_{i-1}, x_i y^{\bar e, i} + x^{\bar e, i} y_i, x_{i+1}y_{i+1}, \ldots, x_{m+\widetilde m}y_{m+\widetilde m}),\]
where $1 \le i \le m$, $p_i$ is the primitive vector on the ray $\rho_i$ of the fan of $X$, an element $e = (e_1, \ldots, e_m, \tilde e_{m+1}, \ldots, \tilde e_{m + \widetilde m})$ in $\mathfrak{R}_i = \mathfrak{R}_{i,0} + \ZZ^{\widetilde m}$ is a Demazure root corresponding to this ray, and $x^{\bar e, i}$ denotes the monomial $\!\!\prod\limits_{\substack{1 \le j \le m\\j \ne i}} \!\! x_j^{\langle p_j, e\rangle} \!\!\!\!\!\! \prod\limits_{m < j \le m + \widetilde m} \!\!\!\!\! x_j^{\tilde e_j}$. 

3) If $r = 0$, then $\barX = X = \AA^n$ and
\[\bar x \cdot \bar y = (x_1 + y_1, \ldots, x_n+y_n).\]
\end{theorem}

\begin{proof}
1) By the same argument as in Theorem~\ref{comult_theor}, $X$ is toric and the multiplication comes from the action of the torus~$T$ on $X$ with an open orbit. By Lemma~\ref{lifttorus_lem}, it is coherent with the diagonal action of the torus~$\barT = (\KK^\times)^{m + \widetilde m}$ on~$\barX = \AA^m \times (\KK^\times)^{\widetilde m}$. Taking a base point $\bar x_0 = (1, \ldots, 1)$, we obtain componentwise multiplication on~$\barX$ in accordance with Construction~\ref{monoid_constr}. By Lemma~\ref{act_mon_coh_lem}, it is coherent with monoid structure of rank~$n$ on~$X$. 

2) Arguing as in Theorem~\ref{comult_theor}, it can be shown that $X$ is toric and any monoid structure on $X$ with the group of invertible elements $G = \GG_m^{n-1} \times \GG_a$ is given by the action of $G = \Ker \chi^e \times \GG_a$ on $X$, where $e \in \mathfrak{R}_i$ is a Demazure root and $\GG_a$-action corresponds to $e$. By Lemma~\ref{liftGa_lem}, the action of $\GG_a$ on $X$ is coherent with the action of $\GG_a$ on $\barX$ corresponding to the Demazure root~$\bar e$. By Lemma~\ref{liftGm_lem}, the action of $\Ker \chi^e$ on $X$ is coherent with the action of $\Ker \chi^{\bar e}$ on $\barX$. According to~\cite[Proposition~6]{AK}, $\barX$ is a group embedding of their product $\Ker \chi^{\bar e} \times \GG_a$ for any base point~$\bar x_0$ in the open orbit, and whence we have a monoid structure on~$\barX$. By Lemma~\ref{act_mon_coh_lem}, it is coherent with monoid structure on~$X$ defined by the action of the group~$G$ with the base point~$x_0 = \pi(\bar x_0)$. 

The derivation corresponding to the Demazure root $\bar e = (\langle p_1, e\rangle, \ldots, \langle p_m, e\rangle, \tilde e_1, \ldots, \tilde e_m)$ of~$\barX$ is $\pa_{\bar e} = x^{\bar e, i} \frac{\pa}{\pa x_i}$, and an element $(t, \alpha) \in \Ker \chi^{\bar e} \times \GG_a$ acts as follows: 
\[(t, \alpha) \cdot \bar x = 
\bigl(t_1x_1, \ldots, t_{i-1}x_{i-1}, \; t_i(x_i + \alpha x^{\bar e, i}), \, 
t_{i+1}x_{i+1}, \ldots, t_{m+\widetilde m}x_{m+\widetilde m}\bigr).\] 
According to Construction~\ref{monoid_constr}, take a point $\bar x_0 = (1, \ldots, 1, \underset{i}{0}, 1, \ldots, 1) \in \barX$ in the open orbit and consider the embedding $\Ker \chi^{\bar e} \times \GG_a \hookrightarrow \barX$: 
\[(t, \alpha) \mapsto (t, \alpha) \cdot \bar x_0 = (t_1, \ldots, t_{i-1}, t_i\alpha, t_{i+1}, \ldots, t_{m+\widetilde m}) =: \bar y = (y_1, \ldots, y_{m + \widetilde m}).\]
Since $t \in \Ker \chi^{\bar e}$, we have $t_i^{-1} t^{\bar e, i} = 1$, whence the substitution $t_j = y_j$ for $j \ne i$ and $t_i\alpha = y_i$ gives the required multiplication on~$\barX$:
\[
\bar x \cdot \bar y = (x_1y_1, \ldots, y^{\bar e, i} x_i + y_i x^{\bar e, i}, \ldots, x_{m+\widetilde m}y_{m+\widetilde m}).
\]

3) The same argument used in Theorem~\ref{comult_theor} shows that $X = \AA^n$ with the vector monoid structure. Since the total coordinate space~$\barX$ in this case coincides with~$X$, we have the same multiplication on it. 
\end{proof}

\begin{remark}
One can verify that the multiplications in Theorems~\ref{comult_theor} and~\ref{Cox_theor} are coherent. For instance, the multiplication in Theorem~\ref{Cox_theor}, 2) corresponds to the following comultiplication $\KK[x_1, \ldots, x_{m+\widetilde m}]^{H_X} \to \KK[x_1, \ldots, x_{m+\widetilde m}]^{H_X} \otimes \KK[y_1, \ldots, y_{m+\widetilde m}]^{H_X}$ in terms of evaluating homomorphisms: if $\bar x, \bar y \in \barX$ have coordinates $x_i, y_i$, $1 \le i \le m + \widetilde m$, then a function
\[\chi^{\bar u}(\bar x) = \!\!\prod\limits_{1 \le j \le m} \!\! x_j^{\langle p_j, u\rangle} \!\!\!\! \prod\limits_{m < j \le m+\widetilde m} \!\!\!\!\! x_j^{\tilde e_j}\] evaluated at the point $\bar x \cdot \bar y$ equals
\begin{multline*}
\chi^{\bar u}(\bar x \cdot \bar y) 
= \bigl(
x_i y^{\bar e, i} + y_i x^{\bar e, i}
\bigr)^{\langle p_i, u\rangle}
\!\!\prod_{\substack{1 \le j \le m \\ j \ne i}} (x_jy_j)^{\langle p_j, u\rangle} 
\!\!\!\!\! \prod\limits_{m < j \le m+\widetilde m} \!\!\!\! (x_j y_j)^{\tilde e_j} = 
\\ =
\bigl(
y_i^{-1} y^{\bar e, i} + x_i^{-1} x^{\bar e, i}
\bigr)^{\langle p_i, u\rangle}
\!\!\prod_{1 \le j \le m} (x_jy_j)^{\langle p_j, u\rangle}
\!\!\!\!\! \prod\limits_{m < j \le m+\widetilde m} \!\!\!\! (x_j y_j)^{\tilde e_j}
=
\bigl(\chi^{\bar e}(y) + \chi^{\bar e}(x)\bigr)^{\langle p_i, u\rangle} \chi^{\bar u}(x) \chi^{\bar u}(y).
\end{multline*}
Since the isomorphism $\KK[X] \to \KK[x_1, \ldots, x_{m+\widetilde m}]^{H_X}$ is given by $\chi^u \mapsto \chi^{\bar u}$, we obtain the comultiplication $\KK[X] \to \KK[X] \otimes \KK[X]$ from Theorem~\ref{comult_theor}, 2). 
\end{remark}

\begin{example}
\label{dim3_Coxexample}
Let $X$ be the hypersurface $\{ab = cd\} \subseteq \AA^4$, see Examples~\ref{dim3_1example},~\ref{dim3_multexample}, and~\ref{dim3_2example}. The multiplication on~$\barX = \AA^4$ corresponding to $e = (-1, l_2, l_3) \in \mathfrak{R}_1$ is given by
\[
(x_1, x_2, x_3, x_4) \cdot (y_1, y_2, y_3, y_4) = (x_1y_2^{l_2}y_3^{l_3-1}y_4^{l_2+l_3} + x_2^{l_2}x_3^{l_3-1}x_4^{l_2+l_3}y_1, x_2y_2, x_3y_3, x_4y_4). 
\]
It is easy to verify that this monoid structure on~$\AA^4$ is coherent with monoid structure on~$X$ from Example~\ref{dim3_multexample}. Indeed, let $x = \pi(x_1, x_2, x_3, x_4), y = \pi(y_1, y_2, y_3, y_4)$ be the images in~$X$. According to the Cox construction given in Example~\ref{dim3_2example} we obtain 
\begin{gather*}
a(x*y) = (x_1y_2^{l_2}y_3^{l_3-1}y_4^{l_2+l_3} + x_2^{l_2}x_3^{l_3-1}x_4^{l_2+l_3}y_1)\cdot(x_3y_3) = \hspace{0.3\textwidth} \\ 
= x_1 x_3 \cdot (y_2y_4)^{l_2} \cdot (y_3y_4)^{l_3} +  (x_2x_4)^{l_2} \cdot (x_3x_4)^{l_3} \cdot y_1 y_3 = \\
\hspace{0.4\textwidth} = a(x) b(y)^{l_2} c(y)^{l_3} + b(x)^{l_2} c(x)^{l_3} a(y),\\
b(x*y) = (x_2y_2) \cdot (x_4y_4) = b(x)b(y), \quad
c(x*y) = (x_3y_3) \cdot (x_4y_4) = c(x)c(y),\\
d(x*y) = (x_1y_2^{l_2}y_3^{l_3-1}y_4^{l_2+l_3} + x_2^{l_2}x_3^{l_3-1}x_4^{l_2+l_3}y_1)\cdot(x_2y_2) = \hspace{0.3\textwidth} \\ 
= x_1 x_2 \cdot (y_2y_4)^{l_2+1} \cdot (y_3y_4)^{l_3-1} + (x_2x_4)^{l_2+1} \cdot (x_3x_4)^{l_3-1} \cdot y_1 y_2 = \\
\hspace{0.4\textwidth} = d(x) b(y)^{l_2+1} c(y)^{l_3-1} + b(x)^{l_2+1} c(x)^{l_3-1} d(y).
\end{gather*}
\end{example}

\section{Affine surfaces}
\label{surf_sec}

The following theorems are the main results of the paper. 

\begin{theorem}
\label{surf_mult_theor}
Let $X$ be a normal affine algebraic surface with commutative monoid structure of rank~$r$. Then $X$ is toric, so $\KK[X] = \bigoplus\limits_{u \in S_X} \KK\chi^u$, and we have the following classification of monoid structures up to isomorphism. 

1) If $r = 2$, then $X$ is endowed with the toric monoid structure and the comultiplication $\KK[X] \to \KK[X] \otimes \KK[X]$ is given by \[\chi^u \mapsto \chi^u \otimes \chi^u \;\text{ for any } u \in S_X.\]

2) If $r = 1$, then the comultiplication $\KK[X] \to \KK[X] \otimes \KK[X]$ is given by
\[\chi^u \mapsto \chi^u \otimes \chi^u (1 \otimes \chi^e + \chi^e \otimes 1)^{\langle p_i, u\rangle} \;\text{ for any } u \in S_X,\] 
where $p_i$ is the primitive vector on a ray of the fan of $X$ and $e \in \mathfrak{R}_i$ is a Demazure root corresponding to this ray.

3) If $r = 0$, then $X = \AA^2$ and $X$ has the vector monoid structure with the comultiplication
\[\chi^u \mapsto \chi^u \otimes 1 + 1 \otimes \chi^u \;\text{ for any } u \in \Zgezero^2.\]
\end{theorem}

\begin{proof}
Follows immediately from Theorem~\ref{comult_theor} for $n = 2$. 
\end{proof}

\begin{theorem}
\label{surf_Cox_theor}
Let $X$ be a normal affine algebraic surface admitting a commutative monoid structure of rank~$r$. Then $X$ is toric, and any multiplication on~$X$ is coherent with a multiplication on~$\barX$. Up to isomorphism of monoid structures on~$X$ and~$\barX$, the product on~$\barX$ is given by the following formulas in Cox coordinates $x_1, x_2$. 

1) If $r = 2$, then $\barX = \AA^2$ or $\barX = X = \AA^1 \times \KK^\times, \,(\KK^\times)^2$, and
\[(x_1, x_2) \cdot (y_1, y_2) = (x_1y_1, x_2y_2).\]

2) If $r = 1$ and $\barX = \AA^2$, then
\begin{gather*}
(x_1, x_2) \cdot (y_1, y_2) = (x_1y_2^{\langle p_2, e\rangle} + x_2^{\langle p_2, e\rangle}y_1, x_2y_2), \quad e \in \mathfrak{R}_1,\\
(x_1, x_2) \cdot (y_1, y_2) = (x_1y_1, x_2y_1^{\langle p_1, e\rangle} + x_1^{\langle p_1, e\rangle}y_2), \quad e \in \mathfrak{R}_2,
\end{gather*}
where $p_1, p_2$ are the primitive vectors on rays of the fan of~$X$ and $\mathfrak{R}_1, \mathfrak{R}_2$ are the sets of Demazure roots of $X$ corresponding to these rays. 

3) If $r = 1$ and $\barX = X = \AA^1 \times \KK^\times$, then
\[(x_1, x_2) \cdot (y_1, y_2) = (x_1y_2^e + x_2^ey_1, x_2y_2), \; e \in \ZZ.\]

4) If $r = 0$, then $\barX = X = \AA^2$ and \[(x_1, x_2) \cdot (y_1, y_2) = (x_1 + x_2, y_1 + y_2).\]
\end{theorem}

\begin{proof}
Cases 1), 2) and 4) follow immediately from Theorem~\ref{Cox_theor}. In case $\barX = \AA^1 \times \KK^\times$, the cone of~$X$ has one ray~$\rho_1$, whence $X = \AA^1 \times \KK^\times$ and we can assume that $p_1 = (1,0)$. Then the set of Demazure roots consists of vectors $(-1, e)$, $e \in \ZZ$, the corresponding LND $\pa_e$ equals $x_2^e\frac{\pa}{\pa x_1}$, and an element $(t, \alpha) \in \Ker \chi^{(-1, e)} \times \GG_a$ acts on $(x_1, x_2)$ via $(t^e (x_1 + \alpha x_2), t x_2)$. Identifying $(t, \alpha)$ with $(t^e \alpha, t)$, we obtain the required monoid structure on $X$. 
\end{proof}

\begin{example}
\label{cone_multexample}
Let us list all monoid structures of rank~$1$ corresponding to $p_1 = (0,1)$ for the quadratic cone $X = \{ac = b^2\} \subseteq \AA^3$, see Examples~\ref{cone_1example} and~\ref{cone_2example}. According to Theorem~\ref{surf_mult_theor}, the comultiplication is given by $\chi^u \mapsto \chi^u \otimes \chi^u (1 \otimes \chi^e + \chi^e \otimes 1)^{u_2}$. Recall that $a=\chi^{(1,0)}$, $b=\chi^{(1,1)}$ and $c=\chi^{(1,2)}$, thus for a Demazure root $e = (l, -1) \in \mathfrak{R}_1$ we have $\chi^e(x) = a(x)^{l + 1} b(x)^{-1}$. Then the coordinates $a, b, c$ of the product $x*y$ are given by
\begin{gather*}
a(x*y) = a(x) a(y),\\
b(x*y) = b(x) a(y)^{l + 1} + a(x)^{l + 1} b(y),\\
c(x*y) = c(x) a(y)^{2l + 1} + a(x)^{2l + 1} c(y) + 2 a(x)^l b(x) a(y)^l b(y).
\end{gather*}
By Theorem~\ref{surf_Cox_theor}, the multiplication on~$\barX = \AA^2$ corresponding to $e = (l, -1) \in \mathfrak{R}_1$ is given by
\[
(x_1, x_2) \cdot (y_1, y_2) = (x_1y_2^{2l+1} + x_2^{2l+1}y_1, x_2y_2). 
\]
It is easy to verify that this multiplication on~$\AA^2$ is coherent with the above multiplication on~$X$. Indeed, for $x = \pi(x_1, x_2)$ and $y = \pi(y_1, y_2)$ we have
\begin{gather*}
a(x*y) = (x_2y_2)^2 = x_2^2 y_2^2 = a(x) a(y),\\
b(x*y) = (x_1y_2^{2l+1} + x_2^{2l+1}y_1) \cdot (x_2y_2) = x_1x_2y_2^{2l+2} + x_2^{2l+2}y_1y_2 = b(x)a(y)^{l+1} + a(x)^{l+1}b(y),\\
c(x*y) = (x_1y_2^{2l+1} + x_2^{2l+1}y_1)^2 = x_1^2 y_2^{4l+2} + x_2^{4l+2} y_1^2 + 2x_1 x_2^{2l+1} y_1y_2^{2l+1} = \\
= c(x)a(y)^{2l+1} + a(x)^{2l+1}c(y) + 2 a(x)^l b(x) a(y)^l b(y).
\end{gather*}
\end{example}

\medskip

For a two-dimensional cone $\sigma$ with two rays, denote by $\tau\colon N_\QQ \to N_\QQ$ the linear map which swaps the primitive vectors $p_1$, $p_2$ on the rays, and by $\tau^*\colon M_\QQ \to M_\QQ$ the induced dual map. We obtain the following corollary of~Proposition~\ref{isomorph_dimn_prop}.

\begin{corollary}
\label{isomorph_prop}
Consider the monoid structures of rank~$1$ on $X$ listed in Theorems~\ref{surf_mult_theor} and~\ref{surf_Cox_theor}.

1) Let the cone $\sigma$ be two-dimensional. 
If $\tau(N) \ne N$, then the monoid structures are pairwise non-isomorphic. 
If $\tau(N) = N$, then $\tau^*\colon \mathcal{R}_2 \to \mathcal{R}_1$ is a bijection, the monoid structures corresponding to Demazure roots $e \in \mathcal{R}_2$ and $\tau^*(e) \in \mathcal{R}_1$ are isomorphic, and monoid structures corresponding to $e \in \mathcal{R}_2$ are pairwise non-isomorphic. 

2) For the one-dimensional cone $\sigma$, all the monoid structures listed in case~3) of Theorem~\ref{surf_Cox_theor} are isomorphic. 
\end{corollary}

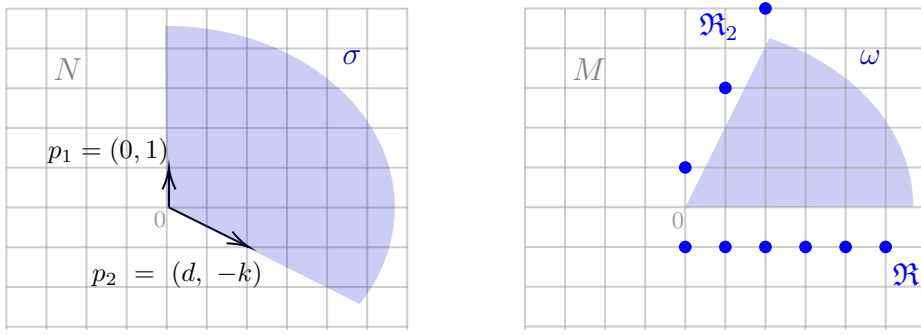
\begin{center}

\tikzset{every picture/.style={line width=0.75pt}} 

\begin{tikzpicture}[x=0.75pt,y=0.75pt,yscale=-1,xscale=1]

\draw  [draw opacity=0] (120.69,20.6) -- (321.08,20.6) -- (321.08,182.11) -- (120.69,182.11) -- cycle ; \draw  [color={rgb, 255:red, 155; green, 155; blue, 155 }  ,draw opacity=0.5 ] (120.69,20.6) -- (120.69,182.11)(140.69,20.6) -- (140.69,182.11)(160.69,20.6) -- (160.69,182.11)(180.69,20.6) -- (180.69,182.11)(200.69,20.6) -- (200.69,182.11)(220.69,20.6) -- (220.69,182.11)(240.69,20.6) -- (240.69,182.11)(260.69,20.6) -- (260.69,182.11)(280.69,20.6) -- (280.69,182.11)(300.69,20.6) -- (300.69,182.11)(320.69,20.6) -- (320.69,182.11) ; \draw  [color={rgb, 255:red, 155; green, 155; blue, 155 }  ,draw opacity=0.5 ] (120.69,20.6) -- (321.08,20.6)(120.69,40.6) -- (321.08,40.6)(120.69,60.6) -- (321.08,60.6)(120.69,80.6) -- (321.08,80.6)(120.69,100.6) -- (321.08,100.6)(120.69,120.6) -- (321.08,120.6)(120.69,140.6) -- (321.08,140.6)(120.69,160.6) -- (321.08,160.6)(120.69,180.6) -- (321.08,180.6) ; \draw  [color={rgb, 255:red, 155; green, 155; blue, 155 }  ,draw opacity=0.5 ]  ;
\draw    (201.8,120.6) -- (201.8,102.6) ;
\draw [shift={(201.8,100.6)}, rotate = 450] [color={rgb, 255:red, 0; green, 0; blue, 0 }  ][line width=0.75]    (10.93,-3.29) .. controls (6.95,-1.4) and (3.31,-0.3) .. (0,0) .. controls (3.31,0.3) and (6.95,1.4) .. (10.93,3.29)   ;
\draw    (201.8,120.6) -- (240.01,139.71) ;
\draw [shift={(241.8,140.6)}, rotate = 206.57] [color={rgb, 255:red, 0; green, 0; blue, 0 }  ][line width=0.75]    (10.93,-3.29) .. controls (6.95,-1.4) and (3.31,-0.3) .. (0,0) .. controls (3.31,0.3) and (6.95,1.4) .. (10.93,3.29)   ;
\draw  [draw opacity=0][fill={rgb, 255:red, 0; green, 0; blue, 200 }  ,fill opacity=0.2 ] (199.81,29.28) .. controls (200.11,29.28) and (200.4,29.28) .. (200.69,29.28) .. controls (263.54,29.28) and (314.48,70.17) .. (314.48,120.6) .. controls (314.48,138.47) and (308.09,155.14) .. (297.03,169.22) -- (200.69,120.6) -- cycle ;
\draw  [draw opacity=0][fill={rgb, 255:red, 0; green, 0; blue, 200 }  ,fill opacity=0.2 ] (501.32,35.55) .. controls (543.35,48.92) and (573.1,81.81) .. (573.19,120.29) -- (459.4,120.47) -- cycle ;
\draw  [draw opacity=0] (379.4,20.47) -- (579.79,20.47) -- (579.79,181.98) -- (379.4,181.98) -- cycle ; \draw  [color={rgb, 255:red, 155; green, 155; blue, 155 }  ,draw opacity=0.5 ] (379.4,20.47) -- (379.4,181.98)(399.4,20.47) -- (399.4,181.98)(419.4,20.47) -- (419.4,181.98)(439.4,20.47) -- (439.4,181.98)(459.4,20.47) -- (459.4,181.98)(479.4,20.47) -- (479.4,181.98)(499.4,20.47) -- (499.4,181.98)(519.4,20.47) -- (519.4,181.98)(539.4,20.47) -- (539.4,181.98)(559.4,20.47) -- (559.4,181.98)(579.4,20.47) -- (579.4,181.98) ; \draw  [color={rgb, 255:red, 155; green, 155; blue, 155 }  ,draw opacity=0.5 ] (379.4,20.47) -- (579.79,20.47)(379.4,40.47) -- (579.79,40.47)(379.4,60.47) -- (579.79,60.47)(379.4,80.47) -- (579.79,80.47)(379.4,100.47) -- (579.79,100.47)(379.4,120.47) -- (579.79,120.47)(379.4,140.47) -- (579.79,140.47)(379.4,160.47) -- (579.79,160.47)(379.4,180.47) -- (579.79,180.47) ; \draw  [color={rgb, 255:red, 155; green, 155; blue, 155 }  ,draw opacity=0.5 ]  ;
\draw  [draw opacity=0][fill={rgb, 255:red, 0; green, 0; blue, 255 }  ,fill opacity=1 ] (496.25,140.47) .. controls (496.25,138.73) and (497.66,137.32) .. (499.4,137.32) .. controls (501.14,137.32) and (502.55,138.73) .. (502.55,140.47) .. controls (502.55,142.21) and (501.14,143.62) .. (499.4,143.62) .. controls (497.66,143.62) and (496.25,142.21) .. (496.25,140.47) -- cycle ;
\draw  [draw opacity=0][fill={rgb, 255:red, 0; green, 0; blue, 255 }  ,fill opacity=1 ] (516.25,140.47) .. controls (516.25,138.73) and (517.66,137.32) .. (519.4,137.32) .. controls (521.14,137.32) and (522.55,138.73) .. (522.55,140.47) .. controls (522.55,142.21) and (521.14,143.62) .. (519.4,143.62) .. controls (517.66,143.62) and (516.25,142.21) .. (516.25,140.47) -- cycle ;
\draw  [draw opacity=0][fill={rgb, 255:red, 0; green, 0; blue, 255 }  ,fill opacity=1 ] (536.25,140.47) .. controls (536.25,138.73) and (537.66,137.32) .. (539.4,137.32) .. controls (541.14,137.32) and (542.55,138.73) .. (542.55,140.47) .. controls (542.55,142.21) and (541.14,143.62) .. (539.4,143.62) .. controls (537.66,143.62) and (536.25,142.21) .. (536.25,140.47) -- cycle ;
\draw  [draw opacity=0][fill={rgb, 255:red, 0; green, 0; blue, 255 }  ,fill opacity=1 ] (556.25,140.47) .. controls (556.25,138.73) and (557.66,137.32) .. (559.4,137.32) .. controls (561.14,137.32) and (562.55,138.73) .. (562.55,140.47) .. controls (562.55,142.21) and (561.14,143.62) .. (559.4,143.62) .. controls (557.66,143.62) and (556.25,142.21) .. (556.25,140.47) -- cycle ;
\draw  [draw opacity=0][fill={rgb, 255:red, 0; green, 0; blue, 255 }  ,fill opacity=1 ] (476.25,140.47) .. controls (476.25,138.73) and (477.66,137.32) .. (479.4,137.32) .. controls (481.14,137.32) and (482.55,138.73) .. (482.55,140.47) .. controls (482.55,142.21) and (481.14,143.62) .. (479.4,143.62) .. controls (477.66,143.62) and (476.25,142.21) .. (476.25,140.47) -- cycle ;
\draw  [draw opacity=0][fill={rgb, 255:red, 0; green, 0; blue, 255 }  ,fill opacity=1 ] (456.25,100.47) .. controls (456.25,98.73) and (457.66,97.32) .. (459.4,97.32) .. controls (461.14,97.32) and (462.55,98.73) .. (462.55,100.47) .. controls (462.55,102.21) and (461.14,103.62) .. (459.4,103.62) .. controls (457.66,103.62) and (456.25,102.21) .. (456.25,100.47) -- cycle ;
\draw  [draw opacity=0][fill={rgb, 255:red, 0; green, 0; blue, 255 }  ,fill opacity=1 ] (476.25,60.47) .. controls (476.25,58.73) and (477.66,57.32) .. (479.4,57.32) .. controls (481.14,57.32) and (482.55,58.73) .. (482.55,60.47) .. controls (482.55,62.21) and (481.14,63.62) .. (479.4,63.62) .. controls (477.66,63.62) and (476.25,62.21) .. (476.25,60.47) -- cycle ;
\draw  [draw opacity=0][fill={rgb, 255:red, 0; green, 0; blue, 255 }  ,fill opacity=1 ] (496.25,20.47) .. controls (496.25,18.73) and (497.66,17.32) .. (499.4,17.32) .. controls (501.14,17.32) and (502.55,18.73) .. (502.55,20.47) .. controls (502.55,22.21) and (501.14,23.62) .. (499.4,23.62) .. controls (497.66,23.62) and (496.25,22.21) .. (496.25,20.47) -- cycle ;
\draw  [draw opacity=0][fill={rgb, 255:red, 0; green, 0; blue, 255 }  ,fill opacity=1 ] (456.25,140.47) .. controls (456.25,138.73) and (457.66,137.32) .. (459.4,137.32) .. controls (461.14,137.32) and (462.55,138.73) .. (462.55,140.47) .. controls (462.55,142.21) and (461.14,143.62) .. (459.4,143.62) .. controls (457.66,143.62) and (456.25,142.21) .. (456.25,140.47) -- cycle ;

\draw (140,84.2) node [anchor=north west][inner sep=0.75pt]  [font=\footnotesize] [align=left] {$\displaystyle p_{1} =( 0,1)$};
\draw (162,146) node [anchor=north west][inner sep=0.75pt]   [align=left] {{\footnotesize $\displaystyle p_{2} \ =\ ( d,\ -k)$}};
\draw (142.69,44) node [anchor=north west][inner sep=0.75pt]  [color={rgb, 255:red, 128; green, 128; blue, 128 }  ,opacity=1 ]  {$N$};
\draw (401.4,43.87) node [anchor=north west][inner sep=0.75pt]  [color={rgb, 255:red, 128; green, 128; blue, 128 }  ,opacity=1 ]  {$M$};
\draw (561.4,147.02) node [anchor=north west][inner sep=0.75pt]  [color={rgb, 255:red, 0; green, 0; blue, 255 }  ,opacity=1 ]  {$\mathfrak{R}_{1}$};
\draw (465,21.47) node [anchor=north west][inner sep=0.75pt]  [color={rgb, 255:red, 0; green, 0; blue, 255 }  ,opacity=1 ]  {$\mathfrak{R}_{2}$};
\draw (286.8,40.2) node [anchor=north west][inner sep=0.75pt]  [color={rgb, 255:red, 208; green, 2; blue, 27 }  ,opacity=1 ]  {$\textcolor[rgb]{0,0,0.78}{\sigma }$};
\draw (545.2,39.8) node [anchor=north west][inner sep=0.75pt]  [color={rgb, 255:red, 0; green, 0; blue, 200 }  ,opacity=1 ]  {$\textcolor[rgb]{0,0,0.78}{\omega }$};
\draw (192.8,121.2) node [anchor=north west][inner sep=0.75pt]  [font=\scriptsize,color={rgb, 255:red, 155; green, 155; blue, 155 }  ,opacity=1 ]  {$0$};
\draw (451.2,121.6) node [anchor=north west][inner sep=0.75pt]  [font=\scriptsize,color={rgb, 255:red, 155; green, 155; blue, 155 }  ,opacity=1 ]  {$0$};

\end{tikzpicture}

\end{center}

Any two-dimensional strongly convex polyhedral cone $\sigma$ has a normal form, see~\cite[Proposition~10.1.1]{CLS}. Namely, there exists a basis $e_1, e_2 \in N$ such that
\[\sigma = \cone(e_2, de_1 - ke_2),\]
where $d > 0$, $0 \le k \le d$, and $\gcd(d, k) = 1$.

In this basis, $p_1 = (0, 1)$, $p_2 = (d, -k)$, and we have the following series of Demazure roots: 
\begin{gather*}
\mathfrak{R}_1 = \{e_1^{(l)} = (l, -1) \mid l \in \Zgezero\}\\
\mathfrak{R}_2 = \{(l_1, l_2) \mid dl_1 - kl_2 = -1, \, b \ge 0\} = \{e_2^{(l)} = e_2^{(0)} + (d, k)l \mid l \in \Zgezero\}
\end{gather*}

\begin{corollary}
If $k^2 \not\equiv 1$ modulo~$d$, then the monoid structures or rank~$1$ listed in Theorems~\ref{surf_mult_theor} and~\ref{surf_Cox_theor} are pairwise non-isomorphic. If $k^2 \equiv 1$ modulo~$d$, then the monoid structures are isomorphic if and only if they correspond to Demazure roots $e_1^{(l)}$ and $e_2^{(l)}$ for some $l \in \Zgezero$.
\end{corollary}

\begin{proof}
According to Corollary~\ref{isomorph_prop}, we have to investigate whether $\tau\colon N_\QQ \to N_\QQ$ swapping two rays of the cone preserves the lattice~$N$. In the bases~$e_1, e_2$, the matrix of~$\tau$ is equal to~$\begin{pmatrix}k & d \\ \frac{1-k^2}{d} & -k \end{pmatrix}$ and it is integer if and only if $d$ divides $k^2 - 1$. 
\end{proof}

For example, monoid structures listed in Example~\ref{cone_multexample} are all the non-isomorphic commutative monoid structures of rank~$1$ on the quadratic cone $\{ac=b^2\} \subseteq \AA^3$ since $1^2 \equiv 1$ modulo~$2$.


\end{document}